\documentclass[12pt]{amsart} 

\usepackage{amsmath, amsfonts,amssymb}
\usepackage[usenames]{color} 
\usepackage[all]{xy}  
\usepackage[ top=1in, bottom=1in, left=1.5in, right=1.5in ]{geometry}
\usepackage[colorlinks]{hyperref}
\usepackage{fullpage}
\usepackage{mathrsfs}
\usepackage{mathabx}
\allowdisplaybreaks

\theoremstyle{plain} 

\newtheorem{Teo}{Theorem}[section]

\newtheorem{Lemma}[Teo]{Lemma}
\newtheorem{Prop}[Teo]{Proposition}
\newtheorem{Cor}[Teo]{Corollary}

\theoremstyle{definition}
\newtheorem{Def}[Teo]{Definition}
\newtheorem{Rem}[Teo]{Remark}
\newtheorem{Alg}[Teo]{Algorithm}

\newtheorem{Ex}[Teo]{Example}

\newtheorem{Notation}[Teo]{Notation}

\numberwithin{equation}{Teo}
\numberwithin{figure}{Teo}

\usepackage{color}   

\newcommand{\cN}{\mathcal{N}}

\newcommand{\cI}{\mathcal{I}}
\newcommand{\cL}{\mathcal{L}}
\newcommand{\cM}{\mathcal{M}}

\newcommand{\cV}{\mathcal{V}}
\newcommand{\NN}{\mathbb N}  
\newcommand{\ZZ}{\mathbb Z}
\newcommand{\FF}{\mathbb F}
\newcommand{\QQ}{\mathbb Q}
\newcommand{\PP}{\mathbb P}

\newcommand{\Hom}{\operatorname{Hom}}

\newcommand{\Supp}{\operatorname{Supp}}

\newcommand{\Spec}{\operatorname{Spec}}

\newcommand{\Length}{\operatorname{length}}

\newcommand{\Sing}{\operatorname{Sing}}
\newcommand{\FSing}{\operatorname{Sing}_F}
\newcommand{\Min}{\operatorname{min}}

\newcommand{\Jac}{\operatorname{Jac}}

\newcommand{\FDer}[1]{\stackrel{#1}{\to}}

\newcommand{\FJac}[2]{\mathfrak{J}_F ({#2}^{#1})}

\begin{document}
\title{$F$-jumping and $F$-Jacobian ideals for hypersurfaces}
\author{Luis N\'u\~nez-Betancourt and Felipe P\'erez}

\maketitle 

\begin{abstract}
We introduce two families of ideals, $F$-jumping ideals and  $F$-Jacobian ideals,
in order to study the singularities of hypersurfaces in positive characteristic.
Both families are defined using the $D$-modules $M_{\alpha}$ that were
introduced by Blickle, Musta\c{t}\u{a} and Smith.
Using strong connections between 
$F$-jumping ideals and generalized test ideals, we give a characterization 
of $F$-jumping numbers for hypersurfaces.
Furthermore, we give an algorithm that determines whether 
certain numbers are $F$-jumping numbers. 
In addition, we use $F$-Jacobian
ideals
to study intrinsic properties of the singularities of hypersurfaces. In particular, we
give conditions for $F$-regularity. 
Moreover, $F$-Jacobian ideals behave similarly to Jacobian ideals of polynomials.
Using techniques developed to study these two new families of ideals, 
we provide
relations among test ideals, generalized test ideals, and generalized
Lyubeznik numbers for hypersurfaces.
\end{abstract}

\tableofcontents 

\section{Introduction}
The aim of this note is to introduce two families of ideals, 
$F$-jumping ideals and $F$-Jacobian ideals. Both families of
ideals measure singularities in positive characteristic. The $F$-jumping
ideal is closely related to the test ideal introduced by Hara
and Yoshida \cite{H-Y}, while the $F$-Jacobian ideal is strongly
connected with the original test ideal defined by Hochster and Huneke \cite{HoHu1}. These two families are related through $M_{\alpha},$
the modules introduced by Blickle, Musta\c{t}\u{a} and Smith \cite{BMS-Hyp}, 
and the theory of $D$-modules and $F$-modules \cite{LyuFMod,Ye}.
Roughly speaking, the $F$-jumping ideal associated to an element $f$ in a regular
$F$-finite ring, $R,$ and to a rational number $\alpha=\frac{r}{p^{e}-1},$
gives the unique simple $D$-submodule of $M_{\alpha}$. If $\alpha=1$, then
$M_{\alpha}\cong R_{f}$ with the usual  $D$-module structure. Moreover,
the unique simple $D$-submodule of $M_\alpha$ is $R$. Under additional
hypotheses, the $F$-Jacobian ideal determines the sum of the simple
$F$-submodules of $R_{f}/R$.  

Several of our results are obtained by investigating properties of $F^e$-submodules of $M_\alpha$ through ideals $I\subset R$ such
that $f^rI\subset I^{[p^e]}.$ We pay particular attention to ideals such that $I=(I^{[p^e]}:f^r).$ This is given by 
Lyubeznik's study of $F$-modules via root maps \cite{LyuFMod}.

\subsection{F-jumping ideals.}
Let $R$ be an $F$-finite regular ring, 
$f$ an element of
$R$ and $\alpha$ a positive real number. The test ideal, $\tau(f^{\alpha}),$
was introduced by Hara and Yoshida  as an analogue of the multiplier
ideal in positive characteristic \cite{H-Y}. The $F$-jumping numbers for $f$
are defined as the positive real numbers $c,$ such that $\tau(f^{c-\epsilon})\not=\tau(f^{c})$
for every $\epsilon >0$. Blickle,  Musta\c{t}\u{a} and Smith showed  that
the $F$-jumping numbers are rational and form a discrete set \cite{BMS-Hyp}. 
These numbers encode important information about the singularity of $f$ (cf. \cite{BFS}).

Let $\alpha$ be the rational number $\frac{r}{p^{e}-1}$, with $r$
a positive integer. $M_{\alpha}$ is $R_{f}\cdot e_{\alpha}$ as an $R$-module,
where $e_{\alpha}$ is a formal symbol. The $D$-module structure
depends on the number $\alpha$ \cite[Remark $2.4$]{BMS-Hyp}. However, 
it does not depend on the on the representation of $\alpha=\frac{r}{p^{e}-1}$. 

One of the main aims of this work is to study the simplicity of $M_\alpha$  as a $D$-module and as a $F^e$-module. 
In Proposition \ref{AlphaSimpleFMod} and Remark \ref{AlphaSimpleDMod}
we show that $N_\alpha=D\cdot f^{\lceil \alpha \rceil} e_\alpha$ is 
the unique simple $D$-submodule of $M_\alpha.$ 
This raises the definition of the
$F$-jumping ideal $\frak{J}(f^{\alpha})$ as the ideal such that 
$$
\frak{J}(f^{\alpha})e_\alpha=R\cdot e_{\alpha}\cap N_{\alpha}.
$$ 
We characterize the simplicity of $M_\alpha$ in the following theorem:

\begin{Teo}\label{MainAlpha}
Let $R$ be an $F$--finite regular domain
and $f\in R$ be a nonzero element.
The following are equivalent:
\begin{itemize}
\item[\rm{(i)}] $\alpha$  
is not an $F$-jumping number;  
\item[\rm{(ii)}] $N_\alpha=M_\alpha$;
\item[\rm{(iii)}] $M_\alpha$ is a simple $D$-module;
\item[\rm{(iv)}] $M_\alpha$ is a simple $F^e$-module;
\item[\rm{(v)}] $\FJac{\alpha}{f}= R.$
\end{itemize}
\end{Teo}
Moreover, the test ideal $\tau(f^{\alpha})$ is the minimal root for $N_\alpha$ and
$\tau(f^{\alpha-\epsilon})$ is the minimal root for $M_\alpha$ (see Propositions \ref{TestIdealMinimalRoot} and \ref{MinimalRoot}). 
Additionally, in Algorithm \ref{Algoritmo} we present a process to decide whether
$\alpha$ is an $F$-jumping number.

\subsection{F-Jacobian ideals.}

Suppose that $S=K[x_1,\ldots,x_n]$ is a polynomial ring over a perfect field, $K,$ and $f\in S.$
The Jacobian ideal is defined by 
$\Jac(f)=(f,\frac{\partial f}{\partial_{x_{1}}},\ldots, \frac{\partial f}{\partial_{x_{n}}}).$ 
This ideal plays a fundamental role in the study of singularity in zero and positive characteristics.
In this case, 
$\Jac(f)=R$ if and only if $R/fR$ is a regular ring. 
Another important property is that
$\Jac(fg)\subset f\Jac(g)+g\Jac(f)$ for $f,g\in S,$ which is given by the Leibniz rule.
The equality in the previous 
containment holds only in specific cases \cite[Proposition $8$]{Eleonore} and it is used to study 
transversality of singular varieties \cite{Eleonore, EleonorePaolo}.

Let $R$ be an $F$-finite regular local ring. We define the $F$-Jacobian
ideal, $J_{F}(f),$ to be the intersection of $R\cdot e_{1}\subseteq M_{1}$
with the sum of the minimal $F$-submodules in $M_{1}$ properly containing
the unique simple $F$-module. 
The $F$-Jacobian ideal behaves
similarly to the Jacobian ideal of a polynomial. 
As the Jacobian ideal, they determine singularity:
\begin{itemize}
\item if $R/fR$ is $F$-regular, then $J_{F}(f)=R$  (Corollary \ref{F-Jac F-SingLocus});
\item if $R/fR$ is $F$-pure, then $R/fR$ is $F$-regular if
and only if $J_{F}(f)=R$ (Corollary \ref{F-Jac TestIdeal}).
\item If $f$ has an isolated singularity and $R/fR$ is $F$-pure, then $J_{F}(f)=R$ if $R/fR$
is $F$-regular, and $J_{F}(f)=m$ otherwise (Proposition \ref{FpureIsolated}).
\end{itemize}
In addition, the $F$-Jacobian ideal also satisfies a Leibniz rule:
$J_{F}(fg)=fJ_{F}(g)+gJ_{F}(f)$ for relatively primes elements $f,g\in R$ (Proposition \ref{F-JacFactors}).
The Leibniz rule in characteristic zero is important in the study of
transversality of singular varieties and free divisors over the complex numbers \cite{Eleonore,EleonorePaolo}.

The $F$-Jacobian ideals behave well with $p^e$-th powers;
$J_{F}(f^{p^{e}})=J_{F}(f)^{[p^{e}]}$ (Proposition \ref{F-JacPrimePower}). 
This is a technical property that 
was essential in several proofs. 
This contrasts with how the Jacobian ideal changes with $p^e$-th powers:
$$\Jac(f^{p^e})=(f^{p^e},p^e\frac{\partial f^{p^e-1}}{\partial x_1},\ldots,p^e\frac{\partial f^{p^e-1}}{\partial x_n})
=f^{p^e}R\neq (f^{p^e},\frac{\partial f}{\partial x_1}^{p^e},\ldots,\frac{\partial f}{\partial x_n}^{p^e})=\Jac(f)^{[p^e]}$$

The $F$-Jacobian ideal can be computed from the test ideal in certan cases and they are strongly related 
(see Proposition \ref{DefFlag}). 
However, they are not the same (see Example \ref{x3y3z3}). 
Moreover, the $F$-Jacobian can be defined for elements such that $R/fR$ is not reduced and satifies properties 
that the test ideal does not (eg.  Propositions  \ref{F-JacFactors} and \ref{F-JacPrimePower}).

Furthermore, we define the $F$-Jacobian ideal for a regular
$F$-finite UFD, $R,$ such that $R_{f}/R$ has finite length as a $D$-module
(Section \ref{SecFJacUFD}). We also define the $F$-Jacobian ideal for a ring which is essentially of finite type over an 
$F$-finite local ring (Section \ref{SecFJacVar}). 
Both definitions agree for rings that belog to the previous families (Corollary \ref{CorDefAgree}).

\subsection{Further consequences and relations}

Using ideas and techniques developed to define and study $F$-jumping ideals and $F$-Jacobian ideals,
we give bounds for the length of $R_f/R$ as an $F$-module.
One of these bounds is in terms of a flag of test ideals previously defined by 
Vassilev \cite{TestIdeals}.
If $R$ is an $F$-finite regular local ring  and $R/fR$ is $F$-pure, there exists a strictly ascending chain of ideals 
$$
fR=\tau_0\subset \tau_1\subset \ldots \subset \tau_\ell=R
$$
such that $(\tau^{[p]}_{i}:\tau_{i})\subset (\tau^{[p]}_{i+1}:\tau_{i+1})$ and $\tau_{i+1}$ is the pullback of the test ideal of $R/\tau_i.$

Another bound is given by the  generalized Lyubeznik numbers.
These are invariants associated to a local ring of equal characteristic defined by the first author and Witt \cite{NuWi}.
These numbers are defined using the $D$-module structure of local cohomology modules. 
For a hypersurface of particular interest is the Lyubeznik number 
$$
\lambda^{\dim(R/fR)}_0(R/fR;L):=\Length_{D(\widehat{R},L)-module}H^{1}_{f\widehat{R}}(\widehat{R}),
$$
where $\widehat{R}$ is the completion of $R$ with respect to its maximal ideal and $L$ is a coefficient field. 
Namely, the relations are:

\begin{Teo}\label{MainBounds}
Let $(R,m,K)$ be an $F$-finite regular local ring and $f\in R$ be an element, such that $R/fR$ is reduced.
Let $f=f_1\ldots f_s$ be a factorization of $f$ into irreducible elements.
Then,
$$
\Length_{F-\hbox{{\tiny mod}}} R_f/R
\leq\Length_R(\tau(f^{1-\epsilon})/\tau_f)+s,
$$
where $\tau_f$ is the pullback of the test ideal, $\tau(R/fR).$
Moreover, if $R/fR$ is $F$-pure, 
  and 
$$0\subset fR=\tau_0\subset \tau_1\subset\ldots\subset \tau_\ell=R$$ is the flag of ideals previously defined, then
$$
\ell\leq \Length_{F-\hbox{{\tiny mod}}} R_f/R \leq\lambda^{\dim(R/fR)}_0(R/fR;L)
$$
for every coefficient field, $L.$
\end{Teo}
Suppose that $R$ is local, $R/fR$ $F$-pure and $K$ perfect.
In this case, $\lambda^{\dim(R/fR)}(R/fR;L)=1$ if and only if $R/fR$ is $F$-regular \cite{NuWi,Manuel}.
This fact and Theorem \ref{MainBounds} say that $\lambda^{\dim(R/fR)}(R/fR;L)$ 
is  measuring ``how far'' an $F$-pure hypersurface is
from being $F$-regular.

\section{Preliminaries}\label{SecPre}

In this section, we review the basic facts that we will need about
tight closure, generalized test ideals, $D$-modules, and $F$-modules. 

\subsection{Tight closure}\label{SecTC}
We recall some definitions in tight closure introduced by Hochster and Huneke \cite{HoHu1, HoHu2},
and mention some properties of this theory \cite{Fedder,LyuKaren}. 
Throughout this section, we assume that $R$ is an $F$-finite ring.

We say that $R$ is \emph{$F$-pure} if for every $R$-module, $M$, 
the morphism induced by the inclusion of $R$ in $R^{1/p}$, $M\otimes_R R\to M\otimes_R R^{1/p}$,
is injective. If $R\to R^{1/p}$ splits, we say that $R$ is \emph{$F$-split.} 
These two properties, $F$-purity and $F$-splitting, are equivalent when $R$ is $F$-finite.
In addition, if $(R,m,K)$ is a regular local ring and $I\subset R$ is an ideal, then
$R/I$ is $F$-pure if and only if $(I^{[p]}:I) \not\subset m^{[p]}$ (Fedder's Criterion, \cite[Theorem $1.12$]{Fedder}).

If $I\subset R$ is an ideal, the \emph{tight closure} $I^*$ of $I$ is the ideal of $R$
consisting of all elements $z\in R$ for which there exists some $c \in R$
that is not in any minimal prime of $R,$ such that
$$
cz^q \in I^{[q]} \hbox{ for all }q = p^e \gg 0.
$$

We say that $R$ is \emph{weakly $F$-regular} if $I=I^*$ for every ideal.
If every localization of $R$ is weakly $F$-regular, we say the $R$ is \emph{$F$-regular}.
In general, tight closure does not commute with localization \cite{Monsky}, and it is unknown if the localization of 
a weakly $F$-regular ring is again a weakly $F$-regular ring. That is why the adjective weakly is used.

$R$ is \emph{strongly $F$-regular} if for all $c\in R$ not in any minimal prime,
there exists some $q=p^e$ such that the morphism of $R$-modules, $R \to R^{1/q}$, that sends $1$ to
$c^{1/q}$ splits. Strong $F$-regularity is preserved after localization.
In a Gorenstein ring, strong and
weak F-regularity are equivalent.

We define the \emph{$F$-singular locus} of $R$ by 
$$\Sing_F(R)=\{P\in \Spec (R)\mid R_P \hbox{ is not }F\hbox{-regular}\}.$$ 
We define \emph{the test ideal of $R$}  by
$$
\tau(R)= \bigcap_{I\subset R}(I:I^*).
$$ 

If $R$ is a Gorenstein ring, we have that
$$
\tau(R)= \bigcap_{I \hbox{ {\tiny parameter ideal}}}(I:I^*).
$$ 
 \cite[Theorem $8.23$]{HoHu1}  \cite[Theorem $18.1$]{Matsumura}.
  
\begin{Rem}\label{F-Sing Gorenstein}
Let $R$ be a reduced ring essentially of finite type over an excellent
local ring of prime characteristic. Let $\tau(R)$ denote the test ideal of $R$. 
We know that  for every multiplicative system $W\subset R$, $W^{-1}\tau(R)=\tau(W^{-1}R)$  \cite[Proposition $3.3$]{KarenParameter}  
 \cite[Theorem $2.3$]{LyuKaren}. It is worth pointing out that, in this case, $\tau (R)$ contains a nonzero-divisor \cite[Theorem 6.1]{HoHu2}.

If $P\in \Sing_F (R)$, then $R_P$ is not $F$-regular, and so $\tau(R_P)=\tau(R)R_P\neq R_P.$ Therefore, $\tau(R)\subset P$ and $P\in\cV(\tau(R)).$
On the other hand, if $P\in\cV(\tau(R))$, then $\tau(R_P)=\tau(R)R_P\neq R_P,$ and then $R$ is not $F$-regular. Hence, $P\in\Sing_F(R).$ 
Therefore, $\Sing_F(R)=\cV(\tau(R)).$
\end{Rem}

\subsection{Generalized test ideals}\label{SecGTI}
Test ideals were generalized by
Hara and Yoshida \cite{H-Y} in the context of pairs $(R,I^{c})$,
where $I$ is an ideal in $R$ and $c$ is a real parameter.
Blickle, Musta\c{t}\u{a}, and Smith \cite{BMS-MMJ} gave an elementary
description of these ideals in the case of a regular $F$-finite ring,
$R$. We give the definition introduced by them:

Given an ideal $I$ in $R$ we denote by $I^{[1/p^{e}]}$
the smallest ideal $J$ such that $I\subseteq J^{[p^{e}]}$  \cite[Definition $2.2$]{BMS-MMJ}.
The existence of the smallest such ideal is a consequence of the flatness
of the Frobenius map in the regular case. 

We recall some properties that we will use often
 $$(IJ)^{1/p^e}\subset I^{[1/p^e]} \cdot J^{[1/p^e]}$$ and 
$$\left(I^{[p^{e}]}\right)^{1/p^s}=I^{[p^{e}/p^s]}\subset \left(I^{[p^{s}]}\right)^{1/p^e}$$
\cite[Proposition $2.4$]{BMS-MMJ}. In addition, $\left(\left( f\right)^{[1/p^e]}\right)^{[p^e]}=D^{(e)}f$
\cite[Proposition $3.1$]{AMBL}, where $D^{(e)}=\Hom_{R^{p^e}}(R,R).$

Given a non-negative number, $c,$ and a nonzero ideal, $I$,
we define the \emph{generalized test ideal with exponent
$c$} by 
\[
\tau(I^{c})=\bigcup_{e>0}(I^{\lceil cp^{e}\rceil})^{[1/p^{e}]},
\]
where $\lceil c\rceil$ stands for the smallest integer greater than or equal to $c.$

The ideals in the union  above form an increasing chain of ideals;
therefore, they stabilize because $R$ is Noetherian. Hence for large
enough $e,$ $\tau(I^{c})=(I^{\lceil cp^{e}\rceil})^{[1/p^{e}]}.$
In particular, $\tau(f^{\frac{s}{p^e}})=\left(f^{s}\right)^{[1/p^e]}$
\cite[Lemma $2.1$]{BMS-Hyp}.

An important property of test ideals is given by Skoda's Theorem \cite[Theorem $2.25$]{BMS-MMJ}: if $I$ is generated by $s$ elements and $c\leq s,$
then $\tau(I^c)=I\cdot\tau(I^{c-1}).$ 

For every nonzero ideal $I$ and
every non-negative number $c$, there exists $\epsilon>0$ such that
$\tau(I^{c})=\tau(I^{c'})$ for every $c<c'<c+\epsilon$  \cite[Corollary $2.16$]{BMS-MMJ}.

A positive real number $c$ is an $\textit{F-jumping number}$ for
$I,$ if $\tau(I^{c})\neq\tau(I^{c-\epsilon})$ for
all $\epsilon>0.$

All $F$-jumping numbers of an ideal $I$ are rational, and they form a discrete set; 
that is, there are no accumulation points of
this set. In fact, they form a sequence with limit infinity  \cite[Theorem $3.1$]{BMS-MMJ}.
Then for every positive number $\alpha,$ 
there is a positive rational number,  $\beta<\alpha,$ such that
$\tau(f^\beta)=\tau(f^\gamma)$ for every $\gamma \in (\beta,\alpha).$
We denote $\tau(f^\beta)$ by $\tau(f^{\alpha-\epsilon}).$

\subsection{$D$-modules} \label{SecDMod}

Given two commutative rings, $A$ and $R,$ such that $A\subset R$,  
we define the ring of $A$-linear differential operators on $R,$ $D(R,A),$ as 
the subring of $\Hom_A(R,R),$ defined inductively as follows.  
The differential operators of order zero are induced by multiplication by elements in $R$ ($\Hom_R(R,R)=R$).   
An element $\theta \in \Hom_A(R,R)$ is a differential operator of 
order less than or equal to $k+1$ if $\theta\cdot r -r\cdot\theta$
is a differential operator of order less than or equal to $k$ for every $r\in R$. 

If $A,$ $B,$ and $R$ are commutative rings rings such that $A\subset B\subset R$, then
$D(R,B)\subset D(R,A)$.

\begin{Notation}
If $A=\ZZ$, we write $D_R$ for $D(R,\ZZ)$, and we write only $D$ if it is clear which ring we are working on.
\end{Notation}
If $M$ is a $D(R,A)$-module, then $M_f$ has a structure of a $D(R,A)$-module,  such that the natural
morphism $M\to M_f$ is a morphism of $D(R,A)$-modules. 
Thus,  $R_f/R$
is an $D(R,A)$-module.

If $R$ is a reduced $F$-finite ring,
we have that 
$D_R=\bigcup_{e\in\NN} \Hom_{R^{p^e}}(R,R)$  \cite{Ye}. We denote $\Hom_{R^{p^e}}(R,R)$ by $D^{(e)}_R.$
Moreover, if $R$ is an $F$-finite domain, then
$R$ is a strongly $F$-regular ring if and only if $R$ is $F$-split and a simple $D_R$-module \cite[Theorem $2.2$]{DModFSplit}.
 
If $R$ is an $F$-finite reduced ring, $W\subset R$ a multiplicative system and $M$ a simple $D_R$-module,
then $W^{-1}M$ is  either zero or a simple $D_{W^{-1}R}$-module. As a consequence, for every $D_R$-module of finite length, $N,$
$$\Length_{D_{W^{-1}R}} W^{-1}N\leq \Length_{D_R}N.$$ 
\subsection{$F^e$-modules}\label{SecFMod}

In this section, we recall some definitions and  properties of the Frobenius functor introduced by Peskine and Szpiro \cite{P-S}.
We assume that $R$ is regular. This allows us to use the theory of $F$-modules introduced by Lyubeznik \cite{LyuFMod}. 

Every morphism of rings $\varphi:R\to S$ defines a functor from $R$-modules to $S$-modules, where  $\varphi^* M=S\otimes_{R} M.$
If $S=R$ and $\varphi$ is the Frobenius morphism, $F M$ denotes $\varphi^* M.$ If $R$ is a regular ring, $F$ is an exact functor.
We denote the  $e$-th iterated  Frobenius functor by $F^e.$ 

\begin{Ex}
If $M$ is the cokernel of a matrix $(r_{i,j}),$ then $F^e M$ is the cokernel of $(r^{p^e}_{i,j}).$ In particular,
if $I\subset R$ is an ideal, then $F^e R/I=R/I^{[p^e]}.$
\end{Ex}

We say that an $R$-module, $\cM$, is an $F^e$-module if there exists an isomorphism of $R$-modules 
$\nu :\cM\to F^e \cM.$

If $M$ is an $R$-module and $\beta :M\to F^e M$ is a morphism of $R$-modules, we  consider
$$
\cM=\lim\limits_\to (M\FDer{\beta} F^e M\FDer{F\beta} F^{2e} M\FDer{F^2 \beta} \ldots).
$$
Then, $\cM$ is an $F^e$-module and  $\cM\FDer{\beta}\cM$ is the structure isomorphism. In this case, we say that $\cM$ is generated by $\beta :M\to F^e_R M$.
If $M$ is a finitely generated $R$-module, we say that $\cM$ is an \emph{$F^e$-finite $F^e$-module}.
If $\beta$ is an injective map, then  $M$ injects into $\cM$. In this case, 
we say that $\beta$ is a \emph{root morphism} and that $M$ is a \emph{root} for $\cM$.

\begin{Ex}\label{ExamplesFMod}
\begin{itemize}
\item[\rm{(i)}] Since $F^e R=R$, we have that $R$ is an $F^e$-module, where the structure morphism $\nu :R\to R$
is the identity.
\item[\rm{(ii)}] For every element $f\in R$ and  $r,e\in \NN$, we take $\alpha=\frac{r}{p^e-1}$ and define $M_{\alpha}$ as the 
$F^e$-finite $F^e$-module that is generated by
$$
R\FDer{f^{r}} R \FDer{f^{p^e r}} R\FDer{f^{p^{2e} r}} \ldots.
$$
\end{itemize}
\end{Ex}
\begin{Rem}
This structure of $M_\alpha$ as $F^e$-module depends of the representation $\frac{r}{p^e-1}.$ However, if we take
another representation $\frac{r'}{p^{e'}-1},$ both structures will induce the same structure as $F^{e\cdot e'}$-module.  
\end{Rem}

We say that $\phi:\cM\to \cN$ is a morphism of $F$-modules if the following diagram commutes:
$$
\xymatrix{
\cM \ar[r]^{\phi} \ar[d]^{\nu_{\cM} }   &   \cN  \ar[d]^{\nu_{\cN} } \\
F\cM  \ar[r]^{F_R\phi}&  F\cN
}
$$

The $F^e$-modules form an Abelian category, and the $F^e$-finite $F^e$-modules form a full Abelian subcategory. Moreover, if $\cM$ is  $F^e$-finite then
$\cM_f$ is also an $F^e$-finite $F^e$-module for every $f\in R$. In addition, 
if $R$ is a local ring, then
every $F^e$ finite $F^e$-module has finite length as $F^e$-module
 and has a minimal root \cite{Manuel,LyuFMod}.
\begin{Ex}
The localization map $R\to R_f $ is a morphism of $F$-modules for every $f\in R$.
\end{Ex}
\begin{Ex}
The quotient of localization map $R\to R_f $ is an $F_R$-finite $F_R$-module for every $f\in R$.
$R_f/R$ is generated by $R/fR\FDer{f^{p-1}} F_R (R/fR)=R/f^pR.$
\end{Ex}

We recall that every $F^e$-submodule $M\subset R_f/R$ is a $D$-module \cite[Examples $5.2$]{LyuFMod}. We have that $R_f/R$
has finite length as an $F$-module, because $R_f/R$ has finite length as a $D$-module. 
Let $R$ be an $F$-finite regular ring.
If $R_f/R$ has finite length as $D$-module, then 
$R_f/R$ has finite length as an $F$-module for every $f\in R$.
Therefore, if $R_f/R$ has finite length as a $D$-module, then
$R_f/R$ has finitely many $F$-submodules \cite{Mel}. 
\section{$F$-jumping ideals and $F$-jumping numbers}\label{SecFJacAlpha}
In this section we define $F$-jumpings ideals and give some basic
properties. In particular, we relate them with the generalized test
ideals and $F$-jumping numbers. 

\begin{Notation}
Throughout this section $R$ denotes an $F$-finite regular domain of characteristic $p>0$ and $\alpha$ 
denotes a rational number whose denominator is not divisible by $p$. 
This means that $\alpha$ has the form
$\frac{r}{p^e-1}$ for some $e.$ 
\end{Notation}

We note that $(p^{a\ell-1}+\ldots +p^e+1)r+\alpha=p^{e\ell}\alpha$ for every $\ell\in\NN.$
Let $R$  be a regular local $F$-finite ring. Then, 
$\phi(x^{1/p^{e}})\in I$ for
all $\phi\in\Hom(R^{1/p^{e}},R$) if and only if $x\in I^{[p^{e}]}.$

\begin{Lemma}\label{Mult p}
$\tau(f^{p\lambda})\subset \tau(f^\lambda)^{[p]}$
\end{Lemma}
\begin{proof}
We have that 
$\frac{\lceil p p^j \lambda\rceil}{p^{j}} \leq \frac{p\lceil p^j \lambda \rceil}{p^{j}} \hbox{ and }
\lim\limits_{j\to\infty}\frac{p\lceil p^j \lambda \rceil}{p^{j}}=p\lambda.$
Then,
$$
\tau(f^{p\lambda})
=\bigcup_{j\in\NN} \left( f^{p\lceil p^j \lambda \rceil}\right)^{[ 1/p^{j}]}
= \bigcup_{j\in\NN} \left(f^{\lceil p^j \lambda \rceil}\right)^{[ 1/p^{j-1}]}
\subset \bigcup_{j\in\NN}\left( \left( f^{\lceil p^j \lambda \rceil}\right)^{[ 1/p^{j}]}\right)^{[p]}
=\tau(f^{\lambda})^{[p]}
$$
by the properties of the ideals $I^{[1/p^j]}$ \cite[Lemma $2.4$]{BMS-MMJ}.
\end{proof}

We will study the $F^{e}$-module $M_{\alpha}=R_{f}e_{\alpha}$ introduced in by Blickle,  Musta\c{t}\u{a} and Smith \cite{BMS-Hyp}.
As an $R_f$-modules, $M_\alpha$ is free of rank one generated by $e_\alpha.$ Here $e_\alpha$ is thought formally as $1/f^{\alpha}.$ 
As an $F^{e}$-module $M_\alpha$ is generated by the morphism
\[
R\overset{f^{r}}{\longrightarrow}F^{e}R=R,
\]
as in Example \ref{ExamplesFMod}.
Moreover, the Frobenius morphism acts on $M_\alpha$ by
\[
F^{e}(\frac{b}{f^{m}}\cdot e_{\alpha})=\frac{b^{p^{e}}}{f^{mp^{e}+r}}\cdot e_{\alpha}.
\]
The structural isomorphism $v^{e}:R^{e}\otimes M_{\alpha}\to M_{\alpha}$ is given by 
\[
v^{e}(a\otimes\frac{b}{f^{m}})=\frac{ab^{p^{e}}}{f^{mp^{e}+r}}\cdot e_{\alpha}
\]
whose inverse morphism is 
\[
(v^{e})^{-1}(\frac{c}{f^{s}}\cdot e_{\alpha})= cf^{s(p^{e}-1)+r}\otimes\frac{1}{f^{s}}\cdot e_{\alpha}.
\]
$M_\alpha$ carries a natural structure as a $D_R$-module, 
which does not depend of the presentation of $\alpha$  \cite[Remark $2.4$]{BMS-Hyp}. 
Given $P\in D_{R}^{se},$ we take 
\[
P\cdot\frac{c}{f^{m}}=(v^{se})^{-1}(P\cdot v^{se}(\frac{c}{f^{m}})),
\]
where $P$ acts in $R^{se}\otimes M_{\alpha}$ by $P(a\otimes u)=P(a)\otimes u$.
We note that 
\[
v^{se}(a\otimes\frac{b}{f^{m}})=\frac{ab^{p^{es}}}{f^{mp^{es}+r(1+p^{e}+\ldots + p^{e(s-1)})}}\cdot e_{\alpha}
\]
and 
\[
(v^{se})^{-1}(\frac{c}{f^{m}})=cf^{(m(p^{e}-1)+r)(1+p^{e}+\ldots + p^{e(s-1)})}\otimes\frac{1}{f^{m}}\cdot e_{\alpha}
\]
\[
=cf^{m(p^{es}-1)+r(1+p^{e}+\ldots + p^{e(s-1)})}\otimes\frac{1}{f^{m}}\cdot e_{\alpha}
\]
thus 
\[
P\cdot\frac{c}{f^{m}}=\frac{P(cf^{m(p^{es}-1)+r(1+p^{e}+\ldots +p^{e(s-1)})})}{f^{mp^{es}+r(1+p^{e}+\ldots p^{e(s-1)})}}\cdot e_{\alpha}.
\]
When $m=0,$ the previous expression is equal to 
\[
P\cdot c = \frac{P(cf^{r(1+p^{e}+\ldots p^{e(s-1)})})}{f^{r(1+p^{e}+\ldots +p^{e(s-1)})}}\cdot e_{\alpha}.
\]
\begin{Def}\label{DefFJacAlpha}
We denote $D_R f^{\lceil\alpha\rceil}\cdot e_{\alpha}$ by $N_\alpha.$
We define the \emph{$F$-jumping ideal associated to $f$ and $\alpha$}  as the ideal $\mathfrak{J}_{F_R} ({f}^{\alpha})$ of $R$ such that
$$\mathfrak{J}_{F_R} ({f}^{\alpha}) e_\alpha  =N_\alpha \cap Re_{\alpha}.$$
If it clear in which ring we are working, we only write $\FJac{\alpha}{f}.$
\end{Def}

\begin{Lemma}\label{alpha+1}
The morphism $M_\alpha\to M_{\alpha+1}$ defined by sending $e_{\alpha}\mapsto f e_{\alpha+1} $ is an isomorphism of $F^e$-modules (as well as $D$-modules). 
In particular, $N_\alpha \cong  N_{\alpha+1}=D_R f^{\lceil\alpha+1\rceil}\cdot e_{\alpha+1}.$
\end{Lemma}
\begin{proof}
We have the following commutative diagram of $R$-modules,
$$
\xymatrix{
0\ar[d] & 0\ar[d] && 0\ar[d]\\ 
R\ar[r]^{f^r} \ar[d]_{f} & R\ar[rr]^{f^{rp^e}} \ar[d]_{f^{p^e}}&& R\ar[rr]^{f^{rp^{2e}}} \ar[d]_{f^{p^{2e}}}&&\ldots\\
R\ar[r]^{f^{r+p^e-1}} \ar[d]& R\ar[rr]^{f^{(r+p^{e}-1)p^e}} \ar[d]&& R\ar[rr]^{f^{(r+p^{e}-1)p^{2e}}} \ar[d]&& \ldots \\
R/f \ar[r]^{f^{r+p^{e}-1}} \ar[d]& R/f^{p^e} \ar[rr]^{f^{(r+p^{e}-1)p^{e}}} \ar[d] && R/f^{p^{2e}} \ar[rr]^{f^{(r+p^{e}-1)p^{2e}}}\ar[d] &&\ldots\\
0 & 0&& 0\\
}
$$
We note that 
$$
R/f \FDer{f^{r+p^{e}-1}}  R/f^{p^e}
$$
is zero map because $r\geq 1$ as $\alpha>0.$
By taking the limits, we obtain that
$M_\alpha\FDer{f} M_{\alpha+1}$ is an isomorphism of $F$-modules, and hence of $D_R$-modules.
\end{proof}

\begin{Prop}
Let $\ell\in\NN$. Then,
$f^{\ell}\FJac{\alpha}{f}
\subset 
\FJac{\alpha+\ell}{f}.
$
\end{Prop}
\begin{proof}
Let $\phi:M_\alpha\to M_{\alpha+\ell}$ be the  morphism of $D_R$-modules given by $e_\alpha\mapsto f^\ell e_{\alpha+\ell}.$ We have that
$\phi$ is an isomorphism by Lemma \ref{alpha+1}.
Then,
$$
f^{\ell} \FJac{\alpha}{f} e_{\alpha+\ell}
= \FJac{\alpha}{f} f^{\ell}e_{\alpha+\ell}
=\phi(\FJac{\alpha}{f}e_\alpha)
$$
$$
=\phi(N_\alpha \cap Re_\alpha)
=\phi(N_\alpha)\cap \phi(R e_\alpha)
=  N_{\alpha+\ell}\cap R f^\ell e_{\alpha+\ell}
$$
$$
= N_{\alpha+\ell} \cap R f^\ell e_{\alpha+\ell}
\subset  N_{\alpha+\ell}\cap Re_{\alpha+\ell}
=\FJac{\alpha+\ell}{f}e_{\alpha+\ell}.
$$
Therefore, we obtain that $f^\ell \FJac{\alpha}{f}\subset \FJac{\alpha+\ell}{f}.$
\end{proof}
\begin{Rem}\label{DModGenIdeal}
Given any ideal $I\subseteq R,$ we consider the  $D$-module
generated by $I,$ $D_R\cdot Ie_{\alpha}=\bigcup_{s\geq1}D^{es}\cdot Ie_{\alpha}$. This
is equal to 
\[
\bigcup_{s\geq0}\left( D^{es}(f^{r(1+\ldots+p^{e(s-1)})}I )/f^{r(1+\ldots+p^{e(s-1)})}\right)\cdot e_{\alpha}.
\]
This $D_R$-module intersects with $Re_{\alpha}$ at the ideal 
\[
\bigcup_{s\geq0}\left( D^{es}(f^{r(1+\ldots+p^{e(s-1)})}I):f^{r(1+\ldots+p^{e(s-1)})}\right)
\]
and contains $I.$
\end{Rem}


We define inductively a sequence of ideals $\cI^j$ associated to $f^\alpha.$
We take $\cI^1=\tau(f^\alpha).$
Given $\cI^j,$ we take $\cI^{j+1} =((\cI^j)^{[p^e]}:f^r).$
We note that $f^r \cI^1=f^r\tau(f^\alpha)=\tau(f^{r+\alpha})=\tau(f^{p^e\alpha})\subset \tau(f^\alpha)^{[p^e]}=\cI^1(f^\alpha) ^{[p^e]},$
where the second equality follows from Skoda's Theorem and Lemma \ref{Mult p}.
We have that $\cI^1\subset \cI^2$ and $f^r\cI^2\subset (\cI^1)^{[p^e]}.$ 
Inductively, we obtain that  $\cI^j\subset \cI^{j+1}$ and $f^r \cI^{j+1}\subset (\cI^j)^{[p^e]}$
for every $j\in\NN.$

Since $R$ is Noetherian, there is an $N\in\NN$ such that $\cI^{N}=\cI^{N+1}.$
By our definition of $\cI^j,$ $\cI^{N}=\cI^{N+j}$ for every $j\geq N.$
\begin{Def}\label{DefFlag}
We call the previous sequence of ideals \emph{the $F$-flag of ideals associated to $f$ and $\alpha$}, and we denote the ideals by
$\cI^j_R(f^\alpha).$ If it is clear in which ring we are taking this flag, we write only $\cI^j(f^\alpha).$
\end{Def}

\begin{Rem}
By the properties of colon ideals and the flatness of Frobenius, we have that $$\cI^j(f^\alpha)=(\tau(f^\alpha)^{[p^{je}]}:f^{r(1+p^e+\ldots +p^{e(j-1)})}).$$
As a consequence, the $F$-flag of ideals associated to $f$ and $\alpha$ depends on the presentation of $\alpha=\frac{r}{p^e-1}$.
However, the union of this flag does not depend on this presentation as the next proposition shows.
\end{Rem}

\begin{Prop}\label{FJacLimitFlag}
$\FJac{\alpha}{f}=\bigcup_j\cI^j(f^\alpha)=\cI^N(f^\alpha)$ for $N\gg 0.$
\end{Prop}
\begin{proof}
By the definition of $\tau(f^\alpha)$ and properties of test ideals in Section \ref{SecGTI},
$\tau(f^{\alpha})^{[p^{se}]}=\left( \left(f^{\lceil p^e\alpha\rceil }\right)^{[1/p^{se}]}\right) ^{[p^{se}]}=D^{se}(f^{\lceil p^{se}\alpha\rceil})$
for $s\geq N$ for a positive integer $N\in\NN.$ 

Then, by Remark \ref{DModGenIdeal},
\begin{align*}
\FJac{\alpha}{f} e_\alpha & =N_\alpha \bigcap Re_\alpha\\
&=D_R f^{\lceil \alpha \rceil}\cdot e_\alpha \bigcap R e_{\alpha}\\
&=\bigcup_{s\geq 0}\left( D^{es}(f^{r(1+\ldots+p^{e(s-1)})}f^{\lceil\alpha\rceil}):f^{r(1+\ldots+p^{e(s-1)})}\right) e_\alpha\\
&=\bigcup_{s\geq N}\left( D^{es}(f^{r(1+\ldots+p^{e(s-1)})}f^{\lceil\alpha\rceil}):f^{r(1+\ldots+p^{e(s-1)})}\right)e_\alpha\\
&=\bigcup_{s\geq N}\left( D^{es}(f^{r(1+\ldots+p^{e(s-1)})+\alpha\rceil}):f^{r(1+\ldots+p^{e(s-1)})}\right)e_\alpha\\
&=\bigcup_{s\geq N}\left( D^{es}(f^{p^{es}\lceil\alpha\rceil}):f^{r(1+\ldots+p^{e(s-1)})}\right)e_\alpha\\
&=\bigcup_{s\geq N}\left( \tau(f^{\alpha})^{[p^{se}]} :f^{r(1+\ldots+p^{e(s-1)})}\right)e_\alpha\\
&=\bigcup_{s\geq N} \cI^s (f^\alpha)e_\alpha.
\end{align*}
\end{proof}

\begin{Prop}\label{FJacAlphaLoc}
Let $W\subset R$ be a multiplicative system. Then, 
$$\mathfrak{J}_{F_{W^{-1}R}} (f^\alpha)=W^{-1}  \FJac{\alpha}{f}$$
\end{Prop}
\begin{proof}
There exists $N$ such that 
$$\mathfrak{J}_{F_R} (f^\alpha)= \cI^N_R (f^{p^e\alpha})=
\cI^{N+1}_{R}(f^{p^e\alpha}).
$$

We have that $\cI^1_R(f^{p^e\alpha})=\tau(f^\alpha)$
and that $\cI^{j+1}_R (f^{p^e\alpha})=(\cI^j_R(f^{p^e\alpha})^{[p^e]}:f^r).$
We have that $W^{-1}\cI^1_R(f^{p^e\alpha})=W^{-1}\tau(f^\alpha)=\tau(f^\alpha W^{-1}R)=\cI^1_{W^{-1}R}(f^{p^e\alpha})$
\cite[Proposition $2.13$]{BMS-MMJ}, 
Thus, 
$W^{-1}\cI^j_R(f^{p^e\alpha})=\cI^j_{W^{-1}R}(f^{p^e\alpha})$
because $ -\otimes_R W^{-1}R$ is a flat functor.
Since  
$\cI^N_{W^{-1}R} (f^{p^e\alpha})=
\cI^{N+1}_{W^{-1}R}(f^{p^e\alpha}),$
we have that 
$$\mathfrak{J}_{F_{W^{-1}R}}(f^{\alpha})=\cI^N_{W^{-1}R} (f^{p^e\alpha})=W^{-1}\mathfrak{J}_{F_{R}} (f^\alpha).$$
\end{proof}
\begin{Prop}\label{FJacAlphaComplete}
Suppose that $R$ is a local ring. Then
$$\mathfrak{J}_{F_{\widehat{R}}} (f^\alpha)=\mathfrak{J}_{F_{R}} (f^\alpha)\widehat{R},$$
where $\widehat{R}$ denotes the completion of $R$ with respect to the maximal ideal.
\end{Prop}
\begin{proof}
We have that $\cI^1_R(f^{p^e\alpha})=\tau(f^\alpha)$
and that $\cI^{j+1}_R (f^{p^e\alpha})=(\cI^j_R(f^{p^e\alpha})^{[p^e]}:f^r)$  \cite[Proposition $2.13$]{BMS-MMJ}.
In addition,    
$$\cI^1_R(f^{p^e\alpha})\widehat{R}=\tau(f^\alpha)\widehat{R}=\tau(f^\alpha \widehat{R} ) \widehat{R}=\cI^1_{\widehat{R}}(f^{p^e\alpha})$$
\cite[Proposition $2.13$]{BMS-MMJ}.
Thus, 
$\cI^j_R(f^{p^e\alpha})\widehat{R}=\cI^j_{\widehat{R}}(f^{p^e\alpha})$
because $ -\otimes_R \widehat{R}$ is a flat functor.
The rest is analogous to the proof of Proposition \ref{FJacAlphaLoc}.
\end{proof}

Every $F^e$-submodule of $M_\alpha$ is given by an ideal $I\subset R$ such that
$I\FDer{f^r}I^{[p^e]}$ makes sense \cite[Corollary $2.6$]{LyuFMod}. This is $f^rI\subset I^{[p^e]}.$ We set
$$
N_I=\lim\limits_{\to} \left( I\FDer{f^r}I^{[p^e]}\FDer{f^{p^er}}I^{[p^{2e}]}\FDer{f^{p^{2e}r}}\ldots\right).
$$

\begin{Lemma}\label{GeneratesSame}
Let $I,J\subset R$ be ideals such that $f^r I\subset I^{[p^e]},$ $f^r J\subset J^{[p^e]}$
and $I\subset J.$ 
Then the $F^e$-submodule $M_\alpha$ generated by  $I$ is equal to the one generated by $J$ if and only  if 
there exits $\ell\in \NN$ such that $f^{r(1 +\ldots + p^{e(\ell-1)})} J\subset I^{[p^{e\ell}]}.$
\end{Lemma}
\begin{proof}
Let $N_I$ and $N_J$ be the $F^e$-submodules of $M_\alpha$ generated by $I$ and $J$ respectively.
We have that $J/I\FDer{f^r}J^{[p^e]}/I^{[p^e]}$ generates the $F^e$-module $N_J/N_I.$
Since $J/I$ is a finitely generated $R$-module, $N_J=N_I$ if and only if there exist
there exists $\ell$ such that $$J/I\FDer{f^{r(1+\ldots + p^{e(\ell-1)})}}J^{[p^{e\ell }]}/I^{[p^{e\ell }]}$$ is the zero morphism.
Therefore, $N_I=N_J$ if and only if there exits $\ell\in \NN$ such that $f^{r(1 +\ldots + p^{e(\ell-1)})} J\subset I^{[p^e]}.$
\end{proof}

\begin{Prop}\label{AlphaFMod}
$D_R f^{\lceil\alpha\rceil}\cdot e_{\alpha}$ is an $F^e$-submodule of $M_\alpha.$ Moreover,
$\FJac{\alpha}{f}\FDer{f^r} \FJac{\alpha}{f}^{[p^e]}$ and
$\tau(f^\alpha)\FDer{f^r} \tau(f^\alpha)^{[p^e]}$ generate
$D_R f^{\lceil\alpha\rceil}\cdot e_{\alpha}.$ 
\end{Prop}
\begin{proof}
By the definition of $\tau(f^\alpha)$ and properties of test ideals in Section \ref{SecGTI},
$\tau(f^{\alpha})^{[p^{se}]}=\left( \left(f^{\lceil p^e\alpha\rceil }\right)^{[1/p^{se}]}\right) ^{[p^{se}]}=D^{se}(f^{\lceil p^{se}\alpha\rceil})$
for $s\geq N$ for an integer $N\in\NN.$ 
Then,
\begin{align*}
D_R\cdot f^{\lceil\alpha\rceil} e_{\alpha}&=\bigcup_{s\geq 0}D^{es}\cdot f^{\lceil\alpha\rceil}e_{\alpha}\\
&=\bigcup_{s\geq N}D^{es}\cdot f^{\lceil\alpha\rceil}e_{\alpha}\\
&=\bigcup_{s\geq N}\left( D^{es}(f^{\lceil r(1+\ldots+p^{e(s-1)})}f^{\lceil\alpha\rceil})/f^{r(1+\ldots+p^{e(s-1)})}\right)\cdot e_{\alpha}\\
&=\bigcup_{s\geq N}\left( D^{es}(f^{\lceil p^{es}\alpha\rceil})/f^{r(1+\ldots+p^{e(s-1)})}\right)\cdot e_{\alpha}\\
&=\bigcup_{s\geq N}\left(  \tau(f^{\alpha})^{[p^{se}]} /f^{r(1+\ldots+p^{e(s-1)})}\right)\cdot e_{\alpha}.
\end{align*}

As $f^r \tau(f^{\alpha})=\tau(f^{r+\alpha})=\tau(f^{p^e\alpha})\subset \tau(f^{\alpha})^{[p^e]}$ by Lemma \ref{Mult p},
it is a root for the $F^{e}$-submodule of $M_\alpha$ generated by the morphism
$$\tau(f^{\alpha}) \FDer{f^r} \tau(f^{\alpha})^{[p^e]},$$ which is
$$
\bigcup_{s} \left(\tau(f^{\alpha})^{[p^{se}]}/f^{r(1+\ldots+p^{e(s-1)})}\right) \cdot e_\alpha.
$$
Hence, $D_R\cdot f^{\lceil\alpha\rceil} e_{\alpha}$ is the $F^e$-submodule of $M_\alpha$ generated by the morphism
$\tau(f^{\alpha}) \FDer{f^r} \tau(f^{[\alpha})^{[p^e]}.$ 

As $f^r\cI^s(f^\alpha)\subset \cI^s(f^\alpha)^{[p^e]}$ for every $s,$ we have that $\cI^s(f^\alpha)$ also generates an $F^e$-submodule of $M_\alpha.$
Since $\cI^{s+1}(f^\alpha)/\cI^s(f^\alpha)$ is the kernel of the morphism $R/\cI^s(f^\alpha)\FDer{f^r}R/\cI^s(f^\alpha)^{[p^e]},$

$$\cI^s(f^\alpha)\FDer{f^r} \cI^{s}(f^\alpha)^{[p^e]}\hbox{ and }\cI^{s+1}(f^\alpha)\FDer{f^r} \cI^{s+1}(f^\alpha)^{[p^e]}$$ generate the same $F^e$-submodule by Lemma \ref{GeneratesSame}.
Therefore, $D_R\cdot f^{\lceil\alpha\rceil} e_{\alpha}$ is the $F^e$-submodule generated by $\FJac{\alpha}{f}\FDer{f^r}\FJac{\alpha}{f}^{[p^e]}$ because $\FJac{\alpha}{f}=\cI^{s}(f^\alpha)$ for $s\gg 0.$

\end{proof}

\begin{Lemma}\label{finI}
Let $I\subset R$ be a nonzero ideal such that $I\subseteq(I^{[p^{e}]}:f^{r}).$
Then $f\in\sqrt{I}$.
\end{Lemma}
\begin{proof}
We have that for every $s\in\NN,$  $I\subset (I^{[p^{se}]}:f^{r(1+p^e+\ldots +p^{e(s-1)})} )$ by flatness of Frobenius and properties of colon ideals.
Thus,
$I R_{f}\subset((I R_{f})^{[p^{se}]}: f^{r(1+p^e+\ldots +p^{e(s-1)})})=(IR_{f})^{[p^{se}]}.$ 
Let $P\subset R$ be a prime ideal of $R$ that does not contain $f.$
We have that $IR_P=I^{[p^{es}]}R_P$ for every $s>0.$
We claim that $IR_P=R_P$ for every $P;$ otherwise, 
\[
IR_P\subset\bigcap_s (IR_P)^{[p^{se}]}\subseteq\bigcap (PR_P)^{p^{se}}=0
\]
and we know that $IR_P\neq 0.$
Therefore, $IR_{f}=R_{f}$ and $f\in\sqrt{I}$.
\end{proof}

\begin{Prop}\label{AlphaSimpleFMod}
For any nonzero $F^e$-submodule $N$ of $M_{\alpha},$
 $N_\alpha\subset N.$ In particular, 
$N_\alpha$ is a simple $F$-module.
\end{Prop}
\begin{proof}
As any two $R$-modules intersect in $R_f,$ we have that there is a
minimal simple $F^e$-submodule $N.$ 
We know that $Re_\alpha \bigcap N=Ie_\alpha,$ where $I=(I^{[p^e]}:f^{r})$  \cite[Corollary $2.6$]{LyuFMod}.
Then $f\in I$ by Lemma \ref{finI}.

There exists an $N$ such that
$\tau(f^{\alpha})^{[p^{se}]}=\left(\left(f^{\lceil p^e\alpha\rceil }\right)^{[1/p^{se}]}\right)^{[p^{se}]}=D^{se}(f^{\lceil p^{se}\alpha\rceil})$  and that
$\tau (f^{\frac{n}{p^{se}}+ \alpha})=\tau(f^\alpha)$ for every $s\geq N$ (cf. Section \ref{SecGTI}). 

\begin{align*}
D_R\cdot f^{n} e_{\alpha} &= \bigcup_{s\geq 0}D^{es}\cdot f^{n}e_{\alpha}\\
 &= \bigcup_{s\geq N}D^{es}\cdot f^{n}e_{\alpha}\\
&=\bigcup_{s\geq N}\left( D^{es}(f^{r(1+\ldots+p^{e(s-1)})}f^{n})/f^{r(1+\ldots+p^{e(s-1)})}\right)\cdot e_{\alpha}\\
&=\bigcup_{s\geq N}\left( D^{es}(f^{n+r(1+\ldots+p^{e(s-1)})})/f^{r(1+\ldots+p^{e(s-1)})}\right)\cdot e_{\alpha}\\
&=\bigcup_{s\geq N}\left( \left(  \left(f^{n+r(1+\ldots+p^{e(s-1)})}\right)^{[1/p^{se}]}\right)^{[p^{es}]} /f^{r(1+\ldots+p^{e(s-1)})}\right)\cdot e_{\alpha}\\
&=\bigcup_{s\geq N}\left( \tau (f^{\frac{n}{p^{es}} + \frac{r(1+\ldots +p^{e(s-1)} )}{p^{es}}})^{[p^{es}]} / f^{r(1+\ldots+p^{e(s-1)})}\right) \cdot e_{\alpha} \\
& \supset  \bigcup_{s\geq0}\left( \tau (f^{\frac{n}{p^{es}}+ \alpha})^{[p^{es}]} / f^{r(1+\ldots+p^{e(s-1)})}\right) \cdot e_{\alpha} \\
&\supset \bigcup_{s\geq N}\left( \tau (f^{\alpha})^{[p^{es}]} / f^{r(1+\ldots+p^{e(s-1)})}\right) \cdot e_{\alpha} \\
&= D_R\cdot f^{\lceil\alpha\rceil} e_{\alpha}\\
&= N_{\alpha}\\
\end{align*}
Therefore, $N_\alpha\subset N.$ 
\end{proof}

\begin{Cor}\label{AlphaSimpleDMod}
$N_\alpha$ is a simple $D$-module.
\end{Cor}
\begin{proof}
Since  $\FJac{\alpha}{f}$ commutes with localization, we have that 
the construction of $N_\alpha$ commutes with localization. 
Suppose that $N\subset N_\alpha$ is a nonzero $D$-submodule.
After localizing at a maximal ideal, $m,$ we have that $(R\setminus m)^{-1} N\subset (R\setminus m)^{-1} N_\alpha.$
It suffices to prove the statement in the local case.

If $R$ is a local ring, 
$N_\alpha$ is a simple $F^e$-module and it is a direct sum of simple $D_R$-modules \cite[Theorem $3.2$ and $5.6$]{LyuFMod}.
Since every two $R$-modules in $R_f$ intersect, $N_\alpha$ must be a simple $D_R$-module.
\end{proof}

\begin{Prop}\label{TestIdealMinimalRoot}
Let $I\subset \tau(f^\alpha)$ be a nonzero ideal such that $I\subset (I^{[p^e]}:f^r).$
Then, $I=\tau(f^\alpha).$
In particular, $\tau(f^\alpha)e_\alpha$ is a minimal root for $N_\alpha.$
\end{Prop}
\begin{proof}
We have that $I\subset (I^{[p^{es}]}:f^{\frac{p^{es}-1}{p^e-1}}r)$ for every $s\in\NN$ by flatness of Frobenius and properties of colon ideals.
There exists $n\in \NN$ such that $f^n\in I$ by Lemma \ref{finI}.
We have that 
$$
f^{n+\frac{p^{es}-1}{p^e-1} r} \in I^{[p^{es}]}, \hbox{ and so }
\left( f^{n+\frac{p^{es}-1}{p^e-1} r} \right)^{[1/p^{es}]} \subset I.
$$ 
In addition, 
$$
n+\frac{p^{es}-1}{p^e-1} r=\frac{(p^e-1)n+(p^{es}-1)r}{p^e-1} > \frac{r}{p^e-1},$$
and 
$$
\lim\limits_{s\to \infty} \frac{(p^e-1)n+(p^{es}-1)r}{p^e-1}=\frac{r}{p^e-1}=\alpha.
$$
There exists an $N$ such that 
$$
\tau(f^\alpha)=\left( f^{n+\frac{p^{eN}-1}{p^e-1} r} \right)^{[1/p^{eN}]} R\subset I \subset \tau(f^\alpha).
$$
Therefore, $I =\tau(f^\alpha).$
The claim about being a minimal root follows because $I\FDer{f^r} I^{[p^e]}$ generates an nonzero $F^e$-submodule of
$N_\alpha,$
which is simple by Proposition \ref{AlphaSimpleFMod}.
\end{proof}

Using ideas analogous to the previous proof, we recover a result previously obtained by Blickle \cite[Proposition $3.5$]{GammaSheaves}

\begin{Prop}\label{MinimalRoot}
Let $\beta=\frac{a}{b}\in\QQ_{>0}.$  If $\alpha>\beta,$ 
then $f^r\tau(f^\beta)\subset \tau(f^\beta)^{[p^e]}$ and it generates $M_\alpha$ as 
$F$-module.
Moreover, $\tau(f^{\alpha-\epsilon})e_\alpha$ is the minimal root of $M_\alpha$ as a $F^e$-module.
\end{Prop}
\begin{proof}
Since $\frac{r}{p^e-1}>\frac{a}{b},$
we have that $br+a>p^e a.$
Then 
$$
f^r\tau(f^\beta)
=\tau(f^{\beta+r})=\tau(f^{\frac{a+br}{p^e-1}})
\subset \tau(f^\frac{p^e a}{p^e-1})
=\tau(f^{p^e \beta})
\subset \tau(f^{\beta})^{[p^e]}.
$$
Thus $\tau(f^\beta)$ generates an $F$-submodule of $M_\alpha.$

Since $\lim\limits_{\ell\to\infty}\frac{r(1+\ldots p^{e(\ell-1)})}{p^{e\ell}}=\alpha,$
we can pick an $\ell\in\NN$ such that $r(1 +\ldots p^{e(\ell-1)})>p^{e\ell}\beta.$

We have that 
{\small
$$
f^{r(1+\ldots p^{e(\ell-1)})}R=
\left(\left( f^{r(1 +\ldots p^{e(\ell-1)})} R \right)^{[p^e]}\right)^{[1/p^e]}\subset
\left(\left( f^{r(1+\ldots p^{e(\ell-1)})}  \right)^{[1/p^e]}\right)^{[p^e]}=
\tau( f^\frac{r(1 +\ldots p^{e(\ell-1)})}{p^{e\ell}})^{[p^e]}.
$$
}
Then $\tau( f^\frac{{r(1 +p^{e(\ell-1)})}}{p^{e\ell}})$
generates $M_\alpha$ as a $F^e$-module by Lemma \ref{GeneratesSame} because $\frac{(p^{e(\ell-1)} +\ldots 1)r}{p^{e\ell}}<\alpha.$ Since 
$ \tau( f^\frac{{p^{e(\ell-1)} +\ldots 1)r}}{p^{e\ell}})\subset \tau(f^\beta),$
$\tau(f^{\beta})$ also generates $M_\alpha$ as a $F^e$-module.

Let $I \subset R$ be an ideal such that $f^rI\subset I,$ and that $f^{r(q^{\ell-1} +\ldots 1) }R\subset I^{[q^\ell]}$
for some $\ell\in\NN.$ Then, 
$$
\left(f^{r(1+\ldots p^{e(\ell-1)}) }\right)^{[1/p^{e\ell}]}\subset I,
$$
and so,
$$
\tau(f^{\alpha-\epsilon})\subset \tau(f^{\frac{r(1+\ldots p^{e(\mu-1)})}{p^{e\mu}}})=
\left(f^{r(1+\ldots p^{e(\mu-1)})}\right)^{[1/p^{e\mu}]}\subset I.
$$
Thus,
$\tau(f^{\alpha-\epsilon})$
is contained in every ideal that generates $M_\alpha.$
Since $\tau(f^{\alpha-\epsilon})$  generates 
$M_\alpha,$ $\tau(f^{\alpha-\epsilon})e_\alpha$ is the minimal root of $M_\alpha$ as an $F^e$-module.
\end{proof}

\begin{proof}[Proof of Theorem \ref{MainAlpha}]
We have that 
$\alpha=\frac{r}{p^{e}-1}$  
is not an $F$-jumping number if and only if 
$\tau(f^\alpha)\neq \tau(f^{\alpha-\epsilon}).$
This happens if and only if
$N_\alpha=M_\alpha$ by Propositions \ref{TestIdealMinimalRoot} and \ref{MinimalRoot}, which is equivalent to $\FJac{f}{\alpha}=R$.
The rest follows  from Proposition \ref{AlphaSimpleFMod} and Corollary \ref{AlphaSimpleDMod}.
\end{proof}

Theorem \ref{MainAlpha} gives the following algorithm to decide whether a number of the form 
$\frac{r}{p^e-1}$ is an $F$-jumping number for $f\in R:$

\begin{Alg}\label{Algoritmo} $ $\\
\emph{Input:} $\alpha=\frac{r}{p^e-1}$ and $f\in R$.\\
\emph{Output:}  $R$ if $\alpha$ is not $F$-jumping number for $f$ and a proper ideal otherwise.\\
\emph{Process:}
Compute $\tau(f^\alpha)$; \\
Take $J_1=\tau(f^\alpha)$;\\
Compute $J_{n+1}=(J^{[p^e]}_n:f^r)$ until there is an $N$ such that $J_N=J_{N+1}$;\\
\emph{Return:} $J_N$.
\end{Alg}






\begin{Ex}
Let $R=\FF_{13}[x,y]$ and $f=x^2+y^3.$ Therefore:
\begin{itemize}
\item[\hbox{\rm{(i)}}] If $\alpha=\frac{11}{12},$ $\tau(f^\alpha)=(x,y)R$,
$J_1 =(x, y)R,$  and $J_2=((x^{13}, y^{13}): f^{11})=R$ because
$f^{11}$ is equal to
$
x^{22} -2x^{20}y^{3}+3x^{18} y^{6} 
-4x^{16} y^{9} +5x^{14} y^{12}-6x^{12}    
y^{15}-6x^{10} y^{18}+5x^{8} y^{21}-4x^{6} y^{24}+3x^{4}y^{27}-2x^{2}    y^{30}+y^{33}.
$
Then, $\frac{11}{12}$ is not an $F$-jumping number.

\item[\hbox{\rm{(ii)}}] If $\alpha=\frac{10}{12},$ $\tau(f^\alpha)=(x,y)R$,
$J_1 =(x,y)R,$  and $J_2=((x^{13}, y^{13}): f^{10})=(x,y)$ because 
$f^{10}$ is equal to
$
x^{20}-3x^{18} y^{3}+6x^{16} y^{6}+3x^{14} y^{9}+2x^{12} y^{12}+5x^{10} y^{15}+2x^{8}  y^{18}+3x^{6} y^{21}+6x^{4}    
y^{24}-3x^{2} y^{27}+y^{30}.
$
Then, $\frac{10}{12}$ is an $F$-jumping number.
\end{itemize}
\end{Ex}



\begin{Rem}
Since $F$-jumping ideals and test ideals commute with localization, we have that
$\sqrt{\FJac{\alpha}{f}}=\sqrt{(\tau(f^\alpha):\tau(f^{\alpha-\epsilon}))}.$
In general,
$\mathfrak{J}_F(f^\alpha)$ is not equal to $(\tau(f^\alpha):\tau(f^{\beta}))$.
Let $R=\FF_7[x,y],$ $f=x^3y^2$ and $\alpha=\frac{4}{6}.$ 
Then, $\beta=\frac{3}{6}$ is the biggest $F$-jumping number smaller than $\alpha.$
We have that $\tau(f^\beta)=xyR$ and $\tau(f^\alpha)=x^2yR$, and so
$(x^2yR:xyR)=xR.$ However,  $\FJac{\alpha}{f}=x^2.$
\end{Rem}

\begin{Prop}\label{LemaPowerP}
$\cI^j (f^{p^e\alpha})=\cI^{j-1} (f^{\alpha})^{[p^e]}$ for $j\geq 2.$
\end{Prop}
\begin{proof}
We will prove that $\cI^j (f^{p^e \alpha})=\cI^{j-1} (f^{\alpha})^{[p^e]}$ for $j\geq 2$ by induction on $j.$

If $j=2,$ 
$$\cI^2(f^{p^e\alpha})=\cI^2(f^{\alpha+r})
=(\tau(f^{\alpha+r})^{[p^e]}:f^{p^e r})
$$
$$
=(\tau(f^{\alpha+r}):f^{ r})^{[p^e]}=(f^r\tau(f^\alpha):f^{ r})^{[p^e]}=\tau(f^\alpha)^{[p^e]}
=\cI^1 (f^{\alpha}).
$$

Suppose that the claim is true for $j.$ Then,
$$
\cI^{j+1}(f^{p^e\alpha})=(\cI^j(f^{p^e\alpha})^{[p^e]}:f^{p^e r})
=(\cI^j(f^{p^e \alpha}):f^{ r})^{[p^e]}
$$
$$
=(\cI^{j-1}(f^{\alpha})^{[p^e]}:f^{ r})^{[p^e]}
=\cI^j (f^{\alpha})^{[p^e]}.
$$
\end{proof}

\begin{Prop}
$\FJac{p\alpha}{f} = \FJac{\alpha}{f}^{[p]}.$
\end{Prop}
\begin{proof}
We first note that $\FJac{p^e\alpha}{f} = \FJac{\alpha}{f}^{[p^e]}$ because
$$
\FJac{p^e\alpha}{f}
=\bigcup_{j} \cI^j(f^{p^e\alpha})
=\bigcup_{j} \cI^j(f^{\alpha})^{[p^e]}
=\FJac{\alpha}{f}^{[p^e]}
$$
by Lemma \ref{LemaPowerP}.

In addition, we have that
$\FJac{\alpha}{f}^{[p]}=(\FJac{\alpha}{f}^{[p^{e+1}]}:f^{pr}).$ Then, $\FJac{\alpha}{f}^{[p]}$ defines the $F^e$-submodule of $M_{p\alpha}$
generated by the morphism
$$
\FJac{\alpha}{f}^{[p]}\FDer{f^{pr}} \FJac{\alpha}{f}^{[p^{e+1}]}.
$$

Since $\FJac{p\alpha}{f}$ defines the unique simple $F$-submodule of $M_{p\alpha}$ by Propositions \ref{AlphaFMod} and \ref{AlphaSimpleFMod}, 
$ \FJac{p\alpha}{f}\subset \FJac{\alpha}{f}^{[p]}.$

Combining these two observations we have that
$$
\FJac{p^e\alpha}{f}\subset \FJac{p^{e-1}\alpha}{f}^{[p]} \subset \ldots \subset\FJac{p\alpha}{f}^{[p^{e-1}]}\subset \FJac{\alpha}{f}^{[p^e]}=\FJac{p^e\alpha}{f}.
$$
Hence,
$\FJac{p\alpha}{f}=\FJac{\alpha}{f}^{[p]}$ because the Frobenius map is faithfully flat.
\end{proof}

\section{$F$-Jacobian ideals}\label{SectionRf}
The $F$-Jacobian ideal of an element $f$ is connected with the minimal
$F$-module of the first local cohomology of $R$ supported at $f$.
In this section we define the $F$-Jacobian ideal and deduce some
of its properties.

\subsection{Definition for unique factorization domains}\label{SecFJacUFD}

\begin{Notation}\label{SOFJ-I}
Throughout this section $R$ denotes an $F$-finite regular UFD of characteristic $p>0$ 
such that $R_f/R$ has finite length as a $D(R,\ZZ)$-module 
for every $f\in R.$
\end{Notation}

This hypothesis is satisfied for every $F$-finite regular local ring and for every $F$-finite 
polynomial ring  \cite[Theorem $5.6$]{LyuFMod}.

\begin{Lemma}\label{Intersection}
Let $S$ be a UFD and $f\in S$ be an irreducible element. Then, $N\cap M\neq 0$ for any nonzero $S$-submodules $M,N\subset S_f/S$.
\end{Lemma}
\begin{proof}
Let 
$a/f^\beta\in M\setminus\{0\}$ and $b/f^\gamma\in N\setminus\{0\}$, where 
$\beta,\gamma\geq 1$.
Since $S$ is a UFD and $f$ is irreducible, we may assume that $\gcd(a,f)=\gcd(b,f)=1$. Then, $\gcd(ab,f)=1$, and so
$ab/f\neq 0$ in $S_f/S.$ We have that $ab/f=bf^{\beta-1}(a/f^\beta)=af^{\gamma-1}(b/f^\gamma)\neq 0$. Then, $ab/f\in N\cap M.$. 
\end{proof}

\begin{Lemma}\label{Bijection}
Let $S$ be a regular ring of characteristic $p>0$, $f\in S$ an element and $\pi:S\to S/fS$ be the quotient morphism. Let 
$$\cI:\{I\subset S \mid I\hbox{ is an ideal}, f\in I, (I^{[p]}:f^{p-1})=I\}$$
and 
$$\cN=\{N\subset S_f/S | N\hbox{ is an }F\hbox{-submodule}\}.$$
Then, the correspondence given by  sending $N$ to $I_N=\pi^{-1}(N\cap R/fR)$ is bijective, with inverse defined by sending the ideal $I\in\cI$ to the 
$F$-module $N_I$ generated by
$I/fS\FDer{f^{p-1}} F(I/fS)= I^{[p]}/f^pS .$
\end{Lemma}
\begin{proof}
Since $\phi:R/fR\FDer{f^{p-1}}R/f^pR$ is a root for $R_f/R$, its $F$-submodules are in correspondence with ideals $J\subset R/fR$ such that
$\phi^{-1}(F(J))= J$  \cite[Corollary $2.6$]{LyuFMod}. We have the following generating morphisms:

\begin{figure}[h!]
$$
\begin{array}{*{18}{c@{\,}}}

0& & 0 & &0&&& \\
\downarrow & & \downarrow & & \downarrow  &&& \\
I/fR & \FDer{f^{p-1}} &  F(IR/fR)&\FDer{f^{p^2-p}}& F^2(IR/fR) & \FDer{f^{p^3-p^2}}&\ldots \\
\downarrow  &  & \downarrow && \downarrow  & &\\
R/fR & \FDer{f^{p-1}} &  R/f^pR&\FDer{f^{p^2-p}} & R/f^{p^2} R &\FDer{f^{p^3-p^2}} &\ldots\\
\downarrow  &  & \downarrow && \downarrow  & &\\
R/I & \FDer{f^{p-1}}  & R/I^{[p]} & \FDer{f^{p^2-p}} & R/I^{[p^2]} & \FDer{f^{p^3-p^2}} &\ldots \\
\downarrow  &  & \downarrow && \downarrow  & &\\
0 && 0 && 0 &&&
\end{array}
$$
\end{figure}

Since $J$ is a quotient $I/fR$ of an ideal, $F(J)=I^{[p]}/f^pR$. Then,
\begin{align*}
I/fR&=\phi^{-1}(I^{[p]}/f^p)\\
&=\{h\in R/fR \mid f^{p-1}h\in I^{[p]}/f^p\}\\
&=\{h\in R \mid f^{p-1}h\in I^{[p]}\}/fR\\
&=(I^{[p]}:f^{p-1})/f\\
\end{align*}
and the result follows.
\end{proof}
\begin{Lemma}\label{WellDefMod}
Let $f\in R$ be an irreducible element. 
Then, there is a unique simple $F$-module in $R_f/R$.
\end{Lemma}
\begin{proof}
Since $R_f/R$ is an $F$-module of finite length, there exists a simple $F$-submodule $M\subset R_f/R$.  
Let $N$ be an $F$-submodule
of $R_f/R$. Since $M\cap N\neq 0$ by Lemma \ref{Intersection} and $M$ is a simple $F$-module, $M=M\cap N$. Hence, $M$ is the only nonzero simple $F$-submodule of 
$R_f/R$.
\end{proof}

\begin{Prop}\label{WellDefIdeal}
Let $g\in R$ be an irreducible element and $f=g^n$ for some integer $n\geq 1$.
Then there exists a unique ideal 
$I\subset R$ such that:
\begin{itemize}
\item[\rm{(i)}] $f\in I$,
\item[\rm{(ii)}] $I\neq f R$,
\item[\rm{(iii)}] $(I^{[p]}:f^{p-1})=I$, and
\item[\rm{(iv)}] $I$ is contained in any other ideal satisfying \rm{(i),(ii)} and \rm{(iii).} 
\end{itemize}
\end{Prop}
\begin{proof}
We note that $R_f/R=R_g/R.$
Let $I$ be the ideal corresponding to the minimal simple $F$-submodule in given in Lemma \ref{WellDefMod}
under the bijection in Lemma \ref{Bijection}. Then,
it is clear from Lemma \ref{Bijection} that $I$ satisfies (i)-(iv).
\end{proof}
\begin{Def}
Let  $g\in R$ be an irreducible element and $f=g^n$ for some integer $n\geq 1$. 
We denote  the minimal  simple submodule of $R_f/R$ by  $\Min_{F_R}(f)$, and we called it the \emph{minimal $F$-module of $f$}.
Let $\sigma: R/fR\to R_f/R$ be the morphism defined by $\sigma([a])=a/f$ which is well defined because $R$ is a domain. 
Since image of $\sigma$ is $R\frac{1}{f},$ we abuse notation and consider $R/fR\subset R_f/R.$
We denote $(\phi\sigma)^{-1} (\Min_F(f)\cap R\frac{1}{f})$ by $J_F(f)$, and we call it the \emph{$F$-Jacobian ideal of $f$.}
If $f$ is a unit, we take $\Min_F(f)=0$ and $J_F(f)=R.$
\end{Def}
\begin{Notation}
If it is clear in which ring we are working, we write  $J_F(f)$ instead of  $J_{F_R}(f)$ and $\Min_{F}(f)$ instead of $\Min_{F_R}(f).$
\end{Notation}

\begin{Prop}\label{DModFMod}
Let  $f\in R$ be an irreducible element.
Then $\Min_F(f)$ is the only simple $D$-submodule of $R_f/R$.
\end{Prop}
\begin{proof}
We claim that $R_f/R$ has only one simple $D_R$-module. Since $R_f/R$ has finite length as a $D$-module, there is a  simple $D$-submodule, $M$. 
It suffices to show that for any other $D_R$-submodule, $N\subset M$. We have that $M\cap N \neq 0$ by Lemma \ref{Intersection}, and so $M=M\cap N\subset N$
because $M$ is a simple $D_R$-module.
Since $\Min_{F_R}(f)$ is an $D_R$-module  \cite[Examples $5.1$ and $5.2$]{LyuFMod}, we have that $M\subset \Min_F(f).$

It suffices to prove that $M$ is an $F$-submodule of $R_f/R.$ 
Since $R/fR$ is a domain, we have that the localization morphism,  $R/fR\to R_m/fR_m,$ is injective. 
Then, 
$
\Supp_R(R/fR)=\Supp_R(J) 
$
for every nonzero ideal $J\subset R/fR$. Then, 
$$\Supp_R(R/fR)=\Supp_R(R/fR\cap N)\subset\Supp_R(N)\subset\Supp_R(R_f/R)=\Supp_R(R/f)$$
for every $R$-submodule of $R_f/R$ by Lemma \ref{Intersection}.
Let $m$ denote a maximal ideal such that $f\in m.$ Thus,
$M_m\neq 0,$ and then, $M_m$ is the only simple $D_{R_m}$-module of $(R_f)_m/R_m.$
Since $R_m$ is a regular local $F$-finite ring, we have that $\Min_{F_{R_m}}(f)$ is a finite direct sum of
simple $D_{R_m}$-modules  \cite[Theorem $5.6$]{LyuFMod}. Therefore, $M_m=\Min_{F_{R_m}}(f)$ by Lemma \ref{Intersection}.

Let $\pi:R\to R/fR$ denote the quotient morphism, and $I=\pi^{-1}(R/fR\cap M)$. We note that
$I\neq fR$ because $R/f\cap M\neq 0$ by Lemma \ref{Intersection}. 
We claim that $(I^{[p]}_m:f)=I_m$ for every maximal ideal. If $f\in m$, 
$$I_m/f=(R_m/fR_m)\cap M_m = (R_m/fR_m)\cap \Min_{F_{R_m}}(f)=J_{F_{R_m}}(f)/f;$$ otherwise,
$I_m=R_m=J_{F_{R_m}}(f)$ because $f$ is a unit in $R_m$
Then, $(I^{[p]} :f^{p-1})=I$ and so $I$ corresponds to an $F_R$-submodule of $R_f/R$, $N_I$ by Lemma \ref{Bijection}. 
Moreover, 
$$
N_I=\lim\limits_\to (I/fR\FDer{f^{p-1}}I^{[p]}/f^pR\FDer{f^{p^2-p}}\ldots).
$$
Since localization commutes with direct limit, we have that for every maximal ideal such that $f\in m$,
$$
M_m=\Min_{F_{R_m}}(f)=\lim\limits_\to (I_m/fR_m\FDer{f^{p-1}}I^{[p]}_m/f^pR_m\FDer{f^{p^2-p}}\ldots)=N_{I_m}=N_I\otimes_R R_m.
$$
Therefore,
$M=N_I$ because $\Supp_R(M)=\Supp_R(R/f)$, and it is an $F$-submodule of $R_f/R.$ Hence, $M=\Min_{F_R}(f).$
\end{proof}
\begin{Rem}If $f\in R$ is an irreducible element, then:
\begin{itemize}
\item[(i)] $\Min(f)=\Min(f^n)$ for every $n\in\NN$ because $R_{f^n}/R=R_f/R,$
\item[(ii)] $J_F(f)$ is the minimal of  the family of ideals $I$ containing properly $fR$ such that $(I:f^{p-1})=I$ by Proposition \ref{WellDefIdeal}.
\item[(iii)] $J_F(f)$ is not the usual Jacobian ideal of $f$. If $S=\FF_3[x,y,z,w]$ and $f=xy+zw,$ we have that the Jacobian of $f$ is  $m=(x,y,z,w)S.$
However, $m\neq (m^{[p]}:f^2)$.
\item[(iv)] $J_F(f)=R$ if and only if $R_f/R$ is a simple $F$-module by the proof of Proposition \ref{WellDefIdeal} and Lemma \ref{Bijection}. 
\item[(v)] $J_F(f)=R$ if and only if $R_f/R$ is a simple $D_R$-module
by Proposition \ref{DModFMod}. 
\end{itemize}
\end{Rem}

\begin{Prop}\label{MinFModFactors}
Let $f_i,\ldots f_\ell\in R$ be irreducible relatively prime elements and $f=f_1\cdots f_\ell$.
Then $\Min_F(f_i)$ is an $F$-submodule of $R_f/R$.
Moreover, all the simple $F$-submodules of $R_f/R$ are $\Min_F(f_1),\ldots,\Min_F(f_\ell).$  
\end{Prop}
\begin{proof}
The morphism $R_{f_i}/R\to R_{f}/R$, induced by the localization map $R_{f_1}\to R_f$, is a morphism of $F$-finite $F$-modules given by the diagram:
$$
\xymatrix{
0\ar[d] &  0\ar[d] & 0\ar[d] \\
R/f_iR \ar[r]^{f^{p-1}_i} \ar[d]^{f^p_1\ldots \hat{f}^p_{i}\cdots f^p_{\ell}} &  R/f^{p}_i R \ar[r]^{f^{p^2-p}_i} \ar[d]^{f^{p^2}_1\ldots \hat{f}^{p^2}_{i}\cdots f^{p^2}_{\ell}}  & R/f^{p^2}_i R  
\ar[d]^{f^{p^3}_1\ldots \hat{f}^{p^3}_{i}\cdots f^{p^3}_{\ell}} \ar[r]^{\quad \hbox{ } f^{p^3-p^2}_i}& \ldots \\
R/fR  \ar[r]^{f^{p-1}} &  R/f^pR   \ar[r]^{f^{p^2-p}} & R/f^{p^2} R \ar[r]^{\quad\hbox{ } f^{p^3-p^2}} &\ldots
}
$$
Then $\Min_F(f_i)$ is a simple $F$-submodule of $R_f/R$. 
Let $N$ be an $F$-submodule of $R_f/R,$ and  
$a/f^{\beta_1}_1\cdots f^{\beta_\ell}_\ell\in N\setminus \{0\}.$  Since $f_i$ is irreducible, we may assume that $\gcd(a,f_i)=1$ and  $\beta_i\neq 0$ 
for some $i=1\ldots,\ell.$
Thus, $a/f_i\in N\cap R_{f_i}/R$ and $a/f_i\neq 0.$ Then, $\Min_F(f_i)\subset N\cap R_{f_i}/R\subset N.$
In particular, if $N$ is a simple $F$-submodule, then $N=\Min_F(f_i).$ 
\end{proof}

\begin{Rem}\label{SumMin}
As a consequence of Lemma \ref{MinFModFactors}, we have that 
$$\Min_F(f_1)\oplus \ldots \oplus \Min_F(f_\ell)\in R_f/R$$
because $R_g\cap R_h=R$ for all elements $g,h\in R$ such that $\gcd(g,h)=1.$
\end{Rem}

\begin{Def}\label{F-Jac Def}
Let 
$f_i,\ldots f_\ell\in R$ be irreducible relatively prime elements, $f=f^{\beta_1}_1\cdots f^{\beta_\ell}_\ell$, and $\pi:R\to R/fR$ be the quotient morphism.
We define $\Min_F(f)$ by $$\Min_F(f_1)\oplus \ldots \oplus \Min_F(f_\ell),$$ and we called it the \emph{minimal $F$-module of $f$}.
Let $\sigma: R/fR\to R_f/R$ be the morphism defined by $\sigma([a])=a/f$ which is well defined because $R$ is a domain. 
Since image of $\sigma$ is $R\frac{1}{f},$ we will abuse notation and consider $R/fR\subset R_f/R.$
We denote $(\phi\sigma)^{-1} (\Min_F(f)\cap R\frac{1}{f})$ by $J_F(f)$, and we call it the \emph{$F$-Jacobian ideal of $f$.}
\end{Def}
\begin{Rem}
In the local case, $\Min_F(f)$ is the intersection homology $D$-modules $\cL(R/f,R)$ previously defined by Blickle   \cite[Theorem $4.5$]{Manuel}.  
\end{Rem}

\begin{Prop}\label{F-JacFactors}
Let  $f,g\in R$ be relatively prime elements. Then, 
$$J_F(fg)=fJ_F(g)+gJ_F(f).$$ Moreover, $fJ_F(g)\cap gJ_F(f)=fg R.$
\end{Prop}
\begin{proof}
We consider $R_f/R$ and $R_g/R$ as $F$-submodules of $R_{fg}/R$, where the inclusion is given by the localization maps, $\iota_f: R_f\to R_{fg}$
and $\iota_g: R_g\to R_{fg}.$
Let $\pi: R\to R/fgR$ and $\rho: R\to R/fR$ be the quotient morphisms. 
The limit of the morphism induced by the diagram
$$
\xymatrix{
0\ar[d] &  0\ar[d] & 0\ar[d] \\
R/fR \ar[r]^{f^{p-1}} \ar[d]^{g} &  R/f^{p} R \ar[r]^{f^{p^2-p}} \ar[d]^{g^{p^2}}  & R/f^{p^2} R  
\ar[d]^{g^{p^3}} \ar[r]^{\quad f^{p^3-p^2}}& \ldots \\
R/fgR  \ar[r]^{(fg)^{p-1}} &  R/f^p g^pR   \ar[r]^{(fg)^{p^2-p}} & R/(fg)^{p^2} R \ar[r]^{ \quad\quad\quad (fg)^{p^3-p^2}} &\ldots
}
$$
is $\iota_f.$ Moreover, under this correspondence
$$
\xymatrix{
0\ar[d] &  0\ar[d] & 0\ar[d] \\
J_F(f)/fR \ar[r]^{f^{p-1}} \ar[d]^{g} &  J_F(f)^{[p^2]}/f^{p} R \ar[r]^{f^{p^2-p}} \ar[d]^{g^{p^2}}  & J_F(f)^{[p^2]}/f^{p^2} R  
\ar[d]^{g^{p^3}} \ar[r] & \ldots \\
gJ_F(f)/fgR  \ar[r]^{(fg)^{p-1}} &  g^pJ_F(f)^{[p]}/f^p g^pR   \ar[r]^{(fg)^{p^2-p}} & g^{p^2}J_F(f)^{[p^2]}/(fg)^{p^2} R 
\ar[r] &\ldots
}
$$
induces the isomorphism of $F$-modules, $\iota_f:\Min_F(f)\to \iota_f(\Min_F(f)).$
We have that $$g(J_F(f))=\pi^{-1}(\Min_F(f)\cap R/fgR)\subset \pi^{-1}(\Min_F(f)\cap R/fgR)=J_F(fg).$$ 
In addition, $$(g^pJ_F(f)^{[p]}:(fg)^{p-1})=gJ_F(f),$$ and it defines $\Min_F(f)$ as a $F$-submodule of $R_{fg}/R.$
Likewise, $$fJ_F(g)\subset J_F(fg)\hbox{, }(f^pJ_F(g)^{[p]}:(fg)^{p-1})=fJ_F(g),$$ 
and it defines $\Min_F(g)$ as a $F$-submodule of $R_{fg}/R.$
Then, 
$$fJ_F(g)+gJ_F(f)\subset J_F(fg).$$
Since $\Min_F(f)\cap\Min_F(g)=0$, we have that $fJ_F(g)\cap gJ_F(g)=fgR$. 

We claim that
$$(f^p J_F(g)^{[p]}+g^pJ_F(f)^{[p]}:f^{p-1}g^{p-1})=fJ_F(g)+gJ_F(f).$$ 

To prove the first containment, take
$$h\in  (f^p J_F(g)^{[p]}+g^pJ_F(f)^{[p]}:f^{p-1}g^{p-1}).$$ 
Then $f^{p-1}g^{p-1}h=f^pv+g^pw$ for some $v\in( J_F(g))^{[p]}$
and $w\in J_F(g)^{[p]}$. Since $f$ and $g$ are relatively prime, 
$f^{p-1}$ divides $w$ and $g^{p-1}$ divides $v$. Thus,
there exist $a,b\in R$ such that $v=g^{p-1}a$ and $w=g^{p-1}b$. Then, $a\in (J_F(g)^{[p]}:g^{p-1})=J_F(g)$ and 
$b\in (J_F(f)^{[p]}:f^{p-1})=J_F(f)$. Since, 
$$f^{p-1}g^{p-1}h=f^pv+g^pw=f^pg^{p-1}a+g^pg^{p-1}b,$$ 
$h=fa+gb\in fJ_F(g)+gJ_F(f)$. 

For the other containment, 
it is straightforward to check that 
$$
fJ_F(g)+gJ_F(f)\subset (f^p J_F(g)^{[p]}+g^p J_F(f)^{[p]}:f^{p-1}g^{p-1}).
$$
Since $N_{fJ_F(g)+gJ_F(f)},$ the $F$-module generated by $fJ_F(g)+gJ_F(f),$ 
contains $\Min_F(f)$ and $\Min_F(g)$, $$\Min_F(f)\oplus\Min_F(g)\subset N_{fJ_F(g)+gJ_F(f)}.$$
Therefore, $J_F(h)\subset fJ_F(g)+gJ_F(f)$  and the result follows. 
\end{proof}
\begin{Prop}\label{F-JacPower}
Let  $\beta,\gamma\in\NN$ be such that $\beta<\gamma$. Then,
$$f^{\gamma-\beta}J_F(f^\beta)\subset J_F(f^\gamma)\subset J_F(f^\beta).$$
\end{Prop}
\begin{proof}
Let $\sigma_\ell: R/f^\ell\to R_f/R$ be the injection defined by sending $[a]\to a/f^\ell.$ We note that the image of $\sigma_\ell$ is $R\frac{1}{f^\ell}$
We have that 
the following commutative diagram,
$$
\xymatrix{
R/f^\beta R \ar[dd]^{f^{\gamma-\beta}} \ar[dr]^{\sigma_\beta}  &  \\
      & R_f/R \\
R/f^\gamma R \ar[ur]_{\sigma_\gamma}&  \\ 
}
$$
Then, $R\frac{1}{f^\beta} \cap \Min_F(f) \subset R\frac{1}{f^\gamma}\cap \Min_F(f),$ and this corresponds to
$$f^{\gamma-\beta}J_F(f^\beta)/f^\gamma R\subset J_F(f^\gamma)/f^\gamma R.$$ Hence, $f^{\gamma-\beta}J_F(f^\beta)\subset J_F(f^\gamma)$ because $f^\gamma$ belongs to both ideals.

The morphism $R\frac{1}{f^\gamma} \cap \Min_F(f^\beta)\FDer{f^{\gamma-\beta}} R\frac{1}{f^\beta} \cap \Min_F (f)$ is well defined and it is equivalent to the morphism
$J_F(f^\gamma)/f^\gamma R \to J_F(f^\beta)/f^\beta$ given by the quotient morphism $R/f^\gamma R\to R/f^\beta R$.
Then, $J_F(f^\gamma)+f^\beta R\subset J_F(f^\beta)$ and the result follows.
\end{proof}
\begin{Rem}
There are examples in which the containment in Proposition \ref{F-JacPower} is strict.
Let $R=\FF_p[x]$  and $f=x.$ In this case, $R_f/R$ is a simple $F$-module.
Then, $J_F(x^\beta)=R$ for every $\beta\in\NN$ and $f^{\gamma-\beta}J_F(f^\beta)\subset J_F(f^\gamma)$
for every $\gamma>\beta.$
\end{Rem}

\begin{Cor}
Let $f,g\in R$ be such that $f$ divides $g$. Then, $J_F(g)\subset J_F(f).$ 
\end{Cor}
\begin{proof}
This follows from Propositions \ref{F-JacPower} and \ref{F-JacFactors}.
\end{proof}

\begin{Prop}\label{F-JacLocalization}
Let $f\in R$ and  $W\subset R$ be a multiplicative system.
Then,
$J_{F_{W^{-1}R}}(f)=W^{-1} J_{F_R}(f).$
\end{Prop}
\begin{proof}
By Proposition \ref{F-JacFactors}, it suffices to prove the claim for $f=g^n,$ where $g$ is an irreducible element.
We note that $g$ is either a unit or a irreducible element in $W^{-1}R$. 
We have that $\Min_{F_{W^{-1}R}}(f)=\Min_{F_{W^{-1}R}}(g)$ is either zero or a simple $F$-module by  Lemma \ref{DModFMod}. 
Then, $\Min_{F_{W^{-1}R}}(f)=W^{-1}\Min_{F_R}(f),$ and so 
\begin{align*}
J_{F_{W^{-1}R}}(f) /f W^{-1}R& =W^{-1}R/f W^{-1}R\cap \Min_{F_ {W^{-1}R}}(f)\\
& =W^{-1}(R/f R\cap\Min_{F_R}(f))\\
& =W^{-1}J_{F_R}(f)/f W^{-1}R,
\end{align*}
and the result follows because $f$ belongs to both ideals.
\end{proof}

\begin{Prop}\label{F-JacPrimePower}
Let $f\in R$. Then, $J_{F_{R^{1/p^e}}}(f) =J_{F_R}(f)R^{1/p^e}$.
Moreover, $J_{F_{R}}(f^{p^e}) =J_{F_R}(f)^{[p^e]}$.
\end{Prop}
\begin{proof}
By Proposition \ref{F-JacFactors}, we may assume that $f=g^n$ where $g$ is irreducible.
Let $q$ denote $p^e$ and $h$ denote the length of $R_f/R$ in the category of $F$-modules. 
Let $G:R^{1/q}\to R$ be the isomorphism defined by $r\to r^{q}.$ 
Under the isomorphism $G$, $R^{1/q}_f/R^{1/q}$ corresponds to $ R_{f^{q}}/R$. Then, the length of $R^{1/q}_f/R^{1/q}$ in the category of $F_{R^{1/q}}$-modules is $h$. 
Let 
$0=M_0\subset\ldots \subset M_h=R_f/R$
be a chain of $F_R$-submodules of $R_f/R$ such that $M_{i+1}/M_i$ is a simple $F_R$-module.
Let $fR=J_0\subset \ldots \subset J_h=R $ be the corresponding chain of ideals under the bijection given in Lemma \ref{Bijection}.
Since $f=g^n$ and $g$ is irreducible, $M_1=\Min_{F_R}(f)$ and $J_1=J_{F_R}(f)$. We note that $(J_i^pR^{1/q}:f^{p-1})=J_iR^{1/q}$ 
 and $J_iR^{1/q}\neq J_{i+1}R^{1/q}$ because $R^{1/q}$ is a faithfully flat $R$-algebra.

Then, we have a strictly ascending chain of ideals 
$$fR^{1/q}=J_0R^{1/q}\subset \ldots \subset J_hR^{1/p}=R^{1/q} $$
that corresponds to a strictly ascending chain of $F_{R^{1/q}}$-submodules of $R^{1/q}_f/R^{1/q}.$

Since $f=(g^{1/q})^{qn}$, $g^{1/q}$ is irreducible and the length of $R^{1/q}_f/R^{1/q}$ is $h$, we have that
$$J_{F_R}(f)R^{1/q}=J_1R^{1/p}=J_{F_{R^{1/q}}}(f).$$
After applying the isomorphism $G$ to the previous equality, we have that
$$  J_{F_{R}}(f)^{[q]}=  G(J_{F_R}(f)R^{1/q})=G(J_{F_{R^{1/q}}}(f)) = J_{F_{R}}(f^q).$$
\end{proof}

\begin{Prop}\label{F-JacFlat}
Let $R\to S$ be a flat morphism of UFDs and let $f\in R$. 
If $S$ is as in Notation \ref{SOFJ-I}, then $J_{F_S}(f)\subset J_{F_R}(f)S$. 
\end{Prop}
\begin{proof}
We may assume that $f=g^\beta$ where $g$ is an irreducible element in $R$ by Proposition \ref{F-JacFactors}.
 Since $S$ is flat, $(J_{F_R}(f)^{[p]}S:f^{p-1})=J_{F_R}(f)S.$ Let $M$ denote the $F_S$-submodule of $S_f/S$
given by $J_{F_R}(f)S$ under the correspondence in Lemma \ref{Bijection}.
If $f$ is a unit in $S,$ then $J_F(f)S=S$ and the result is immediate.
We may assume that $f$ is not a unit in $S$.
Since $J_F(f)\neq fR,$ we can pick $a\in J_F(f) \setminus fR.$
Then, $a=bg^\gamma$ for some $0\leq \gamma<\beta$ and $b\in R$ such that
$\gcd(b,g)=1.$ Then, $R/g\FDer{b} R/g$ is injective, and so $S/gS\FDer{b} S/gS$ is also injective. Thus, $\gcd(b,g)=1$ in $S$.
Hence, $b/g$ is not zero in $S_g/S$. Moreover, $b/g=g^{\beta-\gamma-1}a/f\in M$ and it is not zero.
Let $g_1,\ldots, g_\ell\in S$ irreducible relatively prime elements such that 
$g=g^{\beta_1}_1\cdots g^{\beta_\ell}_1.$ We have that $b/g_i=hb/g\in S_{g_i}/S\cap M\setminus\{0\}$, where
$h=g^{\beta_1}_1\cdots g^{\beta_i-1}_i \cdots g^{\beta_\ell}_1$. Then, $\Min_{F_S}(g_i)\subset M$ and so $\Min_{F_S}(f)\subset M.$
Therefore, $J_{F_S}(f)\subset J_{F_R} (f)S.$
\end{proof}

\begin{Prop}\label{F-JacCompletion}
Suppose that $R$ is a local ring. Let $f\in R$. 
Then
$$
J_{F_{\widehat{R}}}(f)=
J_{F_R}(f)
\widehat{R},
$$
where $\widehat{R}$ denotes the completion of $R$ with respect to the maximal ideal.
\end{Prop}
\begin{proof}
We have that $\Min_{F_{\widehat{R}}}(f)=\Min_{F_R}(f)\otimes_R \widehat{R}$  \cite[Theorem $4.6$]{Manuel}.
Then,
$$
J_{F_{\widehat{R}}}=\left( \widehat{R}/f\widehat{R}\right) \cap \Min_{F_{\widehat{R}}}(f)
= \left( \left( R/fR \right) \cap\Min_{F_R}(f) \right) \otimes_R \widehat{R}
=J_{F_R}(f)
\widehat{R}
$$

\end{proof}

\begin{Lemma}\label{ExtensionPerfectCase}
Let $R=K[x_1,\ldots,x_n]$, where $K$ is a perfect field. 
Let $K\to L$ be an algebraic field extension of $K$, $S=L[x_1\ldots,x_n]$, and $R\to S$ the map induced by the extension.  
Then, $J_{F_S}(f) = J_{F_R}(f)S$. 
\end{Lemma}
\begin{proof}
We can assume that $f=g^\beta$ where $g$ is an irreducible element in $R$ by \ref{F-JacFactors}.
It suffices to show that $J_{F_R}(R)S \subset J_{F_S}(S)$ by Proposition \ref{F-JacFlat}. 

There is an inclusion $\phi:R_f/R\to S_f/S,$  which is  induced by  $R\to S$.
We take $M=(\Min_{F_S}(f))\cap R_f/R.$ We claim that $M$ is a  $D_R$-module of $R_f/R.$
Since $K$ is perfect, 
$$
D_R=\bigcup_{e\in\NN} \Hom_{S^{p^e}}(S,S)=D(R,K)=R[\frac{1}{t!}\frac{\partial ^t}{\partial x_i^t}].
$$
We note that $ D_R\subset D_S$, and that 
for every $m\in R_f/R$, $\phi(\frac{\partial ^t}{\partial x_i^t} m)=\frac{\partial ^t}{\partial x_i^t}\phi(m)$.
As a consequence,   $\frac{\partial ^t}{\partial x_i^t} m\in M$ for every $m\in M.$
Therefore, $M$ is a $D_R$-module.

Let $I=M\cap R/fR.$  We note that $$I=\Min_{F_S}(f)\cap R/fR=(J_{F_S}(f)/fS)\cap R/fR$$ and that $S/fS$ is an integral extension of $R/fR$ because $L$ is an algebraic extension of $K$. 
Let $r\in J_{F_S}(f)/fS$ not zero, and $a_j\in R/fR$ such that $a_0\neq 0$
$$
r^n+a_{n-1}r^{n-1}+\ldots +a_1 r+a_0=0
$$
in $S/fS.$ Then,
$$
r(a_{n-1}r^{n-1}+\ldots +a_1)=-a_0,
$$
and so $ a_0\in I=(J_{F_S}(f)/fS)\cap R/fR,$ and then $M\neq 0.$
Therefore, $\Min_{F_R}(f)\subset M$ and so $J_F(f)/f\subset I.$
Let $\pi: R\to R/fR$ be the quotient morphism.
Then,
$$
J_{F_R}(f)\subset \pi^{-1} (I)=J_{F_S}(f)\cap R,
$$
and $$J_{F_R}(f)S\subset (J_{F_S}(f)\cap R)S\subset J_{F_S}(f).$$
\end{proof}

\begin{Lemma}\label{ExtensionNotPerfect}
Let $R =K[x_1,\ldots,x_n]$, where $K$ is an $F$-finite field. 
Let $L=K^{1/p}$, $S=L[x_1\ldots,x_n]$, and $R\to S$ the map induced by the extension $K\to L$.  
Then $J_{F_S}(f) = J_{F_R}(f)S$. 
\end{Lemma}
\begin{proof}
We have that $R\subset S\subset R^{1/p}.$
Then, by Proposition \ref{F-JacFlat}, 
$$J_{F_{R^{1/p}}}(f) \subset J_{F_S}(f) R^{1/p}\subset (J_{F_R}(f) S) R^{1/p}=J_{F_R}(f) R^{1/p}.$$
Since $J_{F_{R^{1/p}}}(f) =J_{F_R}(f) R^{1/p}$ by Proposition \ref{F-JacPrimePower}, 
$$0= J_{F_S}(f) R^{1/p}/ (J_{F_R}(f) S) R^{1/p}=(J_{F_S}(f) / J_{F_R}(f) S)\otimes_S R^{1/p}.$$
Therefore, $J_{F_S}(f) = J_{F_R}(f) S$ because $R^{1/p}$ is a faithfully flat $S$-algebra.
\end{proof}

\begin{Lemma}\label{ExtensionPerfectClosure}
Let $R =K[x_1,\ldots,x_n]$, where $K$ is an $F$-finite field. 
Let $L$ be the perfect closure of $K$, $S=L[x_1\ldots,x_n]$, and $R\to S$ the map induced by the extension $K\to L$.  
Then $J_{F_S}(f) = J_{F_R}(f)S$. 
\end{Lemma}
\begin{proof}
We may assume that $f=g^n$ for an irreducible $g\in R$ by Proposition \ref{F-JacFactors}.
Let $S^e= K^{1/p^e}[x_1,\ldots,x_n].$
Let $h_1,\ldots, h_\ell$ denote a set of generators for $J_{F_S}(f)$.
In this case, $(J_{F_S}(f)^{[p]}:f^{p-1})=J_{F_S}(f).$ Then there exist $c_{i,j}\in S$ such that
$$
f^{p-1}h_j=\sum c_{i,j} h^p_j.
$$
Since $S=\bigcup_e S^e,$ there exists $N$ such that $c_{i,j},h_j\in S^N.$
Let $I\subset R^N$ be the ideal generated by $h_1,\ldots, h_\ell.$ 
We note that $IS=J_{F_S}(f)$; moreover,   $J_{F_S}(f)\cap S^N=I$ because $S^e\to S$ splits for every $e\in\NN$.

We claim that $(I^{[p]}:f^{p-1})=I$.
We have that $f^{p-1}h_\ell \in I^{[p]}$ by our choice of $N$ and so $I\subset (I^{[p]}:f^{p-1}).$
If $g\in (I^{[p]}:f^{p-1}), $ then $f^{p-1}g\in I^{[p]}\subset J_{F_S}(f)^{[p]}$ and $g\in J_{F_S}(f)\cap S^N=I.$
   
As in the proof of Lemma \ref{ExtensionPerfectCase}, $(J_{F_S}(f)/fS)\cap (S^N/fS^N)\neq 0$ and then $ J_{F_S}(f)\cap S^N=I\neq fS.$
Therefore, $J_{F_{S^N}} (f)\subset I$ by Proposition \ref{WellDefIdeal}. Hence, 
$$
J_{F_{S^N}} (f)S\subset IS=J_{F_{S}} (f)\subset J_{F_{S^N}} (f) S,
$$
and the result follows because 
$$J_{F_{R}} (f)S=(J_{F_{R}} (f)S^N)S=J_{F_{S^N}} (f)S.$$
\end{proof}

\begin{Teo}\label{F-JacAlgExt}
Let $R =K[x_1,\ldots,x_n]$, where $K$ is an $F$-finite field. 
Let $L$ be an algebraic extension of $K$, $S=L[x_1\ldots,x_n]$, and $R\to S$ the map induced by the extension $K\to L$.  
Then $J_{F_S}(f) = J_{F_R}(f)S$. 
\end{Teo}
\begin{proof}
It suffices to show $J_{F_R}(f)S\subset J_{F_S}(f)  $ by Proposition \ref{F-JacFlat}.
Let $K^*$ and $L^*$ denote the perfect closure of $K$ and $L$ respectively. 
Let $R^*=K^*[x_1,\ldots,x_n]$ and $S^*=L^*[x_1,\ldots,x_n].$

Then,
$$
(J_{F_R}(f)S)S^*
J_{F_{  R}} (f) S^*=
(J_{F_{ R }} (f)R^* )S^*=
J_{F_{ R^* }} (f) S^*=
J_{F_{ S^* }} (f) =
J_{F_{ S }} (f)S^*
$$
by Lemma \ref{ExtensionPerfectCase} and \ref{ExtensionPerfectClosure}.
Therefore,
$$
(J_{F_{ R }} (f)S/J_{F_{ S }} (f))\otimes_S S^*=
(J_{F_{ R }} (f)S)S^*/(J_{F_{ S }} (f))S^*=0.
$$
Hence $J_{F_{ R }} (f)S/J_{F_{ S }} (f)=0$ because $S^*$ is a faithfully flat $S$-algebra.
\end{proof}

\begin{Ex}\label{ExNotTotal}
Let  $R=\FF_3[x,y],$ and $f=x^2+y^2$ and $m=(x,y)$.
We have that $(m^{[p]}: f^{p-1})=m$. Then, $J_{F_R}(f)\subset m.$
Let $\FF_3[i] $ the extension of $\FF_3$ by $\sqrt{-1},$ $S=L[x,y]$ and $\phi:R\to S$ 
be the inclusion given by the field extension.
Then, $J_{F_S}(f)=(x,y)S$ by Proposition \ref{F-JacFactors} because $x^2+y^2=(x+iy)(x-iy).$
Since $\phi$ is a flat extension, $J_{F_S}(f)\subset J_{F_R}(f)S.$
Then, $m=R\cap J_{F_S}(f)\subset R\cap J_{F_R}(f)S.$ Hence, $J_F(f)=m.$
\end{Ex}

\begin{Prop}\label{AMBL-Ideals}
Let  $f\in R$ be an irreducible element. Then,
$$
J_F(f)=\bigcap_{\gcd(a,f)=1}
\left(\bigcup_{e\in\NN} \left(\left( \left(f^{p^e-1}a\right)^{[1/p^e]},f\right)^{[p^e]}: f^{p^e-1}\right)\right)
$$
\end{Prop}
\begin{proof}
We have that $\Min_F(f) $
is the intersection of all nonzero $D$-submodules of $R_f/R$ by Proposition \ref{DModFMod}.
In particular, $\Min_F(f)$ is the intersection of all nonzero cyclic $D$-modules generated by elements $a/f\in R_f/R.$
Hence, 
\begin{align*}
J_F(f)/f&=\bigcap_{\gcd(a,f)=1} \left( \left( D\cdot a/f \right)\cap R/f \right) \\
& = \bigcap_{\gcd(a,f)=1}\left( \bigcup_{e\in\NN}\left( D^{(e)}\cdot a/f\cap R/f\right)\right)\\
\end{align*}
We have that $b\in J_F(f)$ if $b/f \in \bigcap_{\gcd(a,f)=1} \bigcup_{e\in\NN}\left( D^{(e)}\cdot a/f\right)$,
so, for every $a\in R$ such that $\gcd(a,f)=1,$   there exists an $e\in\NN$ such that $b/f\in D^{(e)}\cdot a/f.$ Thus, there exists
$\phi\in\Hom_{R^{p^e}}(R,R)$ such that 
$$
\phi(a/f)=1/f^{p^e} \phi(f^{p^e-1} a)=b/f+r
$$
Therefore, after multiplying by $f^{p^e},$ we have that 
$$b\in \bigcap_{\gcd(a,f)=1}\left(\bigcup_{e\in\NN} \left(\left( \left(f^{p^e-1}a\right)^{[1/p^e]},f\right)^{[p^e]}: f^{p^e-1}\right)\right)$$
because $\left((f^{p^e-1}a)^{[1/p^e]}\right)^{[p^e]}=D^{(e)}\left( f^{p^e-1}a\right).$ 

On the other hand, if $$b\in\bigcap_{\gcd(a,f)=1}\bigcup_{e\in\NN} \left( \left( \left( f^{p^e-1}a\right)^{[1/p^e]},f\right)^{[p^e]}: f^{p^e-1}\right),$$ then
for every $a\in R$  such that $\gcd(a,f)=1,$ there exists an $e\in\NN$ and
$\phi\in\Hom_{R^{p^e}}(R,R)$ such that 
$$
f^{p^{e}-1}b=\phi(f^{p^e-1}a)+f^{p^e}r
$$
because $\left(\left(f^{p^e-1}a),f\right)^{[1/p^e]}\right)^{[p^e]}=D^{(e)}(f^{p^e-1}a).$
Therefore, after dividing by $f^{p^e},$ we have that $b/f\in \bigcap_{\gcd(a,f)=1} \bigcup_{e\in\NN}\left( D^{(e)}\cdot a/f\right)$, and
then $b\in J_F(f)$.
\end{proof}
\begin{Teo}\label{F-Jac F-Reg}
Let $f\in R$ be such that $R/fR$ is a $F$-pure ring. If $J_F(f)=R,$ then $R/fR$ is strongly $F$-regular.
\end{Teo}
\begin{proof}
We may assume that $(R,m,K)$ is local because being reduced, $F$-pure, and $F$-regular are local properties for $R/fR$. 
Then, $f$ is irreducible by Proposition \ref{F-JacFactors}.
Since $J_F(f)=R$, for every $a$ such that $\gcd(a,f)=1$ there exits an $e\in\NN$ such that
$R=\left( \left(\left(f^{p^e-1}a\right)^{[1/p^e]},f\right)^{[p^e]}: f^{p^e-1}\right)$ 
by Lemma \ref{AMBL-Ideals}. Then, 
$f^{p^e-1}\in \left(\left(f^{p^e-1}a\right)^{[1/p^e]},f\right)^{[p^e]}.$
Since $f^{p^e-1}\not\in m^{[p^e]}$ for every $e\in\NN$ by Fedder's Criterion ,
$R=\left(f^{p^e-1}a\right)^{[1/p^e]};$ otherwise,  $(\left(f^{p^e-1}a\right)^{[1/p^e]},f)\subset m.$
Then, there exist a morphism $\phi:R \to R^{p^e}$ of $R^{p^e}$-modules such that
$\phi(f^{p^e-1}a)=1.$ Let $\varphi: R/fR\to R/fR$ be the morphism defined by $\varphi([x])=[\phi(f^{p^e-1}x)].$
We note that $\varphi$ is a well defined morphism of $(R/fR)^p$-modules such that $\varphi([a])=1.$ 
Then, $R/fR$ is a simple $D(R/fR)$-module. Hence, $R/fR$ is strongly $F$-regular  \cite[Theorem $2.2$]{DModFSplit}.
\end{proof}
\begin{Rem}
\begin{itemize}
\item The result of the previous theorem is a consequence of a result of Blickle \cite[Corollary $4.10$]{Manuel}, as $R/fR$ is a Gorenstein ring. However,
our proof is different from the one given there.
\item $J_{F_R}(f)=R$ does not imply that $R/fR$ is $F$-pure. Let $K=\hbox{\rm{Frac}}(\FF_2[u])$ be the fraction field of the 
polynomial ring $\FF_2[u]$, $R=K[[x,y]]$, and $f=x^2+uy^2.$ 
Then, $f$ is an irreducible element such that $R/fR$ is not pure because $f\in (x,y)^{[2]}R$.
Let $L=K^{1/2}$, $S=L[[x,y]]$ and $ R\to S$ be the inclusion given by the extension $K\subset L.$
Thus, $f=(x-u^{1/2}y)^2$ in $S,$ and then $J_{F_S}(S)=S\subset J_{F_R} (f)S.$ 
Then, $R=J_{F_S}(S)\cap R=J_{F_R} (f)S\cap R=J_{F_R} (f)$
because $R\to S$ splits. Hence, $J_{F_R}(f)=R$ and $R/fR$ is not $F$-pure. 
\end{itemize}
\end{Rem}

\subsection{Definition for rings essentially of finite type over an $F$--finite local ring.}\label{SecFJacVar}

\begin{Notation}\label{SOFJ-II}
Throughout this section $R$ denotes a ring 
essentially of finite type over an $F$--finite
local ring.
Let $f\in R,$ $\pi:R\to R/fR$ be the quotient morphism. 
If $R/fR$ is reduced, $\tau_f$ denotes the pullback of the test ideal of $R/fR,$
$\pi^{-1}(\tau(R/fR)).$
\end{Notation}

Under the hypotheses on $R$ in Notation \ref{SOFJ-II}, there is an $F$-module and a $D$-module of $R_f/R$
called the intersection homology $\cL(R,R/fR)$ \cite{ManuelThesis,Manuel}. We have that for every maximal ideal
 $m\subset R,$ $(R\setminus m)^{-1}\cL(R,R/fR)=\Min_{F_{R_m}}(f).$

\begin{Def}\label{DefII}
Let $R/fR\subset R_f/R$ be the inclusion morphism $1\mapsto \frac{1}{f}.$ 
We define the $F$-Jacobian, $J_f(f)$ as the pullback to $R$ of $(R/fR)\cap\cL(R,R/fR).$
\end{Def}

\begin{Lemma}\label{Flag}
Suppose that $R/fR$ is reduced.
Let $I^j(f)=(\tau_f^{[p^{j-1}]}: f^{p^{j-1}-1}).$
Then $I^j(f)\subset I^{j+1}(f)$ and $I^{j+1}(f)=(I^{j}(f)^{[p]}:f^{p-1}).$
\end{Lemma}
\begin{proof}
Since, in this case, the test ideal of $R/fR$ commutes with localization, we may assume that $R$ is a local ring.
We have that $\tau_f/fR$ is the minimal root for $\Min_F(f)$ \cite[Theorem $4.6$]{Manuel}. 
Then, $f^{p-1} I^1(f)=f^{p-1}\tau_f\subset \tau_f^{[p^e]}=I^1(f) ^{[p^e]}$.
Thus, $I^1\subset I^2$ and $f^{p-1} I^2\subset I^{[p]}_1.$ 
Moreover, $I^2/fR$ is also a root for $\Min_F(f)$ because
$I^2(f)/I^1(f)$ is the kernel of the map 
$$
R/I^1(f)\FDer{f^{p-1}} R/I^1(f)^{[p]}.
$$
Inductively, we obtain that  $I^j\subset I^{j+1},$ $f^r I^{j+1}\subset I^j$ and that
$I^j/fR$ is  a root for $\Min_F(f)$
for every $j\in\NN$ and the result follows.
\end{proof}

\begin{Prop}\label{DefFlag}
Suppose that $R/fR$ is reduced.
Then, $J_F(f)=\bigcup_j I^j_R(f).$
\end{Prop}
\begin{proof}
We have that $\tau_f \frac{1}{f}$ is the minimal root for
$\cL(R,R/fR).$ Moreover, any ideal $I^j(f)\FDer{f^{p-1}} I^j(f)^{[p]}$ also generates $\cL(R,R/fR)$ as an $F$-module. 
Moreover, 
$$\bigcup_j I^j_R(f)=\cL(R,R/fR)\bigcap R/fR=J_F(f).$$
\end{proof}

\begin{Rem}
In general, we do not have $\tau_f=J_F(f).$
Let $R=K[x],$ where $K$ is any perfect field of characteristic $p>0$.
Let $f=x^2.$
Then, $\tau_f=x R\neq R=J_F(f).$ In addition,
Example \ref{x3y3z3}, shows another situation where $\tau_f\neq J_F(f).$
\end{Rem}
\begin{Rem}\label{LemmaWellDef}
If $R$ is an $F$-finite local ring, then 
$$J_F(f)\FDer{f^{p-1}}J_F(f)^{[p]}$$
is a generating morphism for $\Min_F(f)$
because in this case $\Min_F(f)=\cL(R,R/fR).$
\end{Rem}
\begin{Cor}\label{CorDefAgree}
Let $S$ be a ring that is as in Notation \ref{SOFJ-I} and as in Notation \ref{SOFJ-II}. Let $f\in S.$
Let $J$ denote the $F$-Jacobian ideal of $f$ as in Definition \ref{F-Jac Def} and let $J'$ the $F$-Jacobian
ideal of $f$ as in Definition \ref{DefII}. Then, $J=J'$. 
\end{Cor}
\begin{proof}
We have that in both contexts the $F$-Jacobian ideal commutes with localization. We may assume that $R$ is a regular local $F$-finite
ring. As $J_2=(J_2^{[p]}:f^{p-1})$ and $J_2/fR \FDer{f^{p-1}} J^{[p]}_2/f^p R$ is a root for $\Min_F(f)$ by Lemma \ref{LemmaWellDef},
 we have that $J_1=J_2$.
\end{proof}
\begin{Rem}
As for every maximal ideal $m\subset R,$ $R_m$ is as in Notation \ref{SOFJ-I}, we have that 
\begin{itemize}
\item $f^n J_F(f)\subset J_F(f^n),$
\item $J_F(f^{p^e})=J_F(f)^{[p^e]},$ and
\item if $\gcd(f,g)=1,$ $J_F(f)=fJ_F(g)+gJ_F(f).$
\end{itemize}
because those properties localize.
\end{Rem}

\begin{Prop}
Suppose that $(R,m,K)$ is local. Let 
$(S,\eta,L)$ denote a regular $F$-finite ring.
Let $R\to S$ be a flat local morphism such that the closed fiber $S/mS$ is regular and $L/K$ is separable.
Then, $J_{F_S}(f)=J_{F_R}(f)S.$
\end{Prop}
\begin{proof}
It suffices to proof that $\min_{F_{R}}(f)S=\min_{F_{S}}(f).$
We can assume without loss of generalization tat $R/fR$ is reduced.
We have that $J_{F_{\widehat{R}}}(f)=J_{F_R}(f){\widehat{R}}$
and $J_{F_{\widehat{S}}}(f)=J_{F_S}(f)\widehat{S}$
by Proposition \ref{F-JacCompletion}.
In addition, the induced morphism in the completion $\widehat{R}\to \widehat{S}$ is still a flat local morphism.
Since $J_{F_S}(f)\subset J_{F_R}(f)S$
and $J_{F_{\widehat{S}}}(f)\subset J_{F_{\widehat{R}}}(f) \widehat{S}$ by Proposition \ref{F-JacFlat}, 
$J_{F_{\widehat{R}}}(f)\widehat{S}/J_{F_{\widehat{S}}}(f)=\left( J_{F_R}(f)S/J_{F_S}(f)\right)\otimes_S\widehat{S}.$
Therefore, we can assume that $R$ and $S$ are complete.

We note that $R/fR\to S/fS$ is again a flat local morphism such that the closed fiber $S/mS$ is regular $L/K$ is separable by flat base change.
Then $S/fS$ is reduced and $\tau(R/fR)S=\tau(S/fS)$  \cite[Theorem $7.2$]{HoHu2}, and so $I^j_{F_S}(f)=I^j_{F_R}(f)S.$ Hence, $J_{F_S}(f)=J_{F_R}(f)S$  by Proposition \ref{DefFlag}.
\end{proof}


\begin{Cor}
Suppose that $R$ is a $\ZZ^h$-graded ring. 
Let $f\in R$ be a homogeneous element.
Then, $J_F(f)$ is a homogeneous ideal.
\end{Cor}
\begin{proof}
It suffices to proof that $\Min(f)$ is a $\ZZ^h$-graded submodule of $R_f/R.$
We can assume that $R/fR$ is reduced. We have that
$\tau(R/fR)$ is a homogeneous ideal \cite[Theorem $4.2$]{HoHu3}. This means that $I^j_R(f)$ is a homogeneous ideal for every $j.$
Therefore, $J_F(f)$ is homogeneous and that $\Min_F(f)$ $\ZZ^n$-graded submodule of $R_f/R.$
\end{proof}

\begin{Cor}\label{F-Jac F-SingLocus}
Let $S$ be a ring that is as in Notation \ref{SOFJ-I} or as in Notation \ref{SOFJ-II}. Let $f\in S$ be such that $R/fR$ is reduced.
Then,
$\cV(J_F(f))\subset \FSing(S/fS).$
Moreover, if $S/fS$ is an $F$-pure ring, then
$\cV(J_F(f))= \FSing(S/fS).$
\end{Cor}
\begin{proof}
For every  prime ideal 
$P\in \cV(J_F(f)),$ $J_{F_{S_P}}(f)\neq S_P.$
Since $S_P$ is as in Notation \ref{SOFJ-II}, we have that
$\tau(S_P/fS_P)\subset J_{F_{S_P}}(f)\subset PS_P.$ Then, $S_P$
is not $F$-regular and then $P\in\FSing(S/fS).$

Now, we suppose that $S/fS$ is $F$-pure.
For every  prime ideal 
$P\in \FSing(S/fS),$ $S_P/fS_P$ is not $F$-regular. Then, 
$J_{F_{R_P}}(f)\neq R_P$ by Theorem \ref{F-Jac F-Reg}.
Then, $P\in \cV(J_F(f)).$
\end{proof}

\begin{Lemma}\label{F-Jac F-Pure}
Let $S$ be a ring that is as in Notation \ref{SOFJ-I} and as in Notation \ref{SOFJ-II}. 
Let $f\in S$ be an element and $Q\subset S$ be a prime ideal.
If $S_Q/fS_Q$ is $F$-pure, then $S_Q/J_{F_{S_Q}}(f)$ is $F$-pure.
\end{Lemma}
\begin{proof}
We may replace $S$ by $S_Q.$ 
Since $S/fS$ is $F$-pure, we have that $f^{p-1}\not\in Q^{[p]}$ by Fedder's Criterion. We have that $f^{p-1}\in (J_F(f)^{[p]}:J_F(f))$, and so
$(J_F(f)^{[p]}:J_F(f))\not\subset Q^{[p]}$. Therefore, $S/J_F(f)$ is $F$-pure.
\end{proof}

\begin{Cor}\label{F-Jac TestIdeal}
Let $f\in R.$ 
If $R/fR$ is an $F$-pure ring, then
$J_F(f)= \tau_f.$
\end{Cor}
\begin{proof}
We have that $\sqrt{J_F(f)}=\sqrt{\pi^{-1}(R/fR)}$ by Corollary
\ref{F-Jac F-SingLocus} because $\FSing(R/fR)=\cV(\tau(R/fR))$ in this case.
Since $R/J_F(f)$ is $F$-pure by Lemma \ref{F-Jac F-Pure}, $J_F(f)$ is a radical ideal.
In addition, $\tau(R/fR)$
is a radical ideal  \cite[Proposition $2.5$]{FW}. Hence, $J_F(f)= \tau_f.$
\end{proof}

\subsection{Examples}\label{SecExamples}

\begin{Prop}\label{FpureIsolated}
Let $f\in R$ be an element with an isolated singularity at the maximal ideal $m.$ If $R_m/fR_m$ is $F$-pure, then
\[
J_{F}(f)=\begin{cases}
R & \hbox{if }R/fR\hbox{ is }F-\hbox{regular}\\
m & \hbox{otherwise}
\end{cases}
\]
\end{Prop}
\begin{proof}
Since $R/fR$ has an isolated singularity at $m,$ we have that $J_F(f)R_P=R_P$ for every prime ideal different from $m.$
Then, $m\subset \sqrt{J_F(f)}.$

If $R_m/fR_m$ is $F$-regular, then $R/fR$ is $F$-regular, and so $J_F(R)=R$ by Theorem \ref{F-Jac F-Reg}. 

If $R_m/fR_m$ is not $F$-regular, 
then $J_F(R)\neq R$ by Theorem \ref{F-Jac F-Reg}. 
Then, $m=\sqrt{J_F(f)}.$
Since $R_m/fR_m$ is $F$-pure, we have that $R_m/J_F(f) R_m$ is $F$-pure by Lemma \ref{F-Jac F-Pure}.
Then, $R_m/J_F(f) R_m$ is a reduced ring. Hence, $J_F(f)=m.$
\end{proof}

\begin{Ex}\label{EllipticCurves}
Let $K$ is an $F$-finite field. 
Let  $E$ be an elliptic curve  over $K$.
 We choose a closed immersion of $E$ in  $\PP^2_K$ and set $R = K[x, y, z],$  the completed
homogeneous co-ordinate ring of $\PP^2_K$.
We take $f\in R$ as the cubic
form defining $E.$
We know that $f$ has an isolated singularity at $m=(x,y,z)R.$
If the elliptic curve is ordinary, then $R/fR$ is $F$-pure  
\cite[Proposition $4.21$]{Har}   \cite[Theorem $2.1$]{Bhatt} \cite{BhattSingh}.
We know that $R/fR$ is never an $F$-regular ring \cite[Discussion $7.3b(b)$, Theorem $7.12$]{HoHu3}.
Then, $J_F(f)=m$ by Proposition \ref{FpureIsolated}.
\end{Ex}
\begin{Ex}\label{x3y3z3}
Let $R=K[x,y,z],$ where is an $F$-finite field of characteristic $p>3.$
Let $f=x^{3}+y^{3}+z^{3}\in R,$ and $\pi: R\to R/fR$ be the quotient morphism and $m=(x,y,z)R$. 
We have that $\tau_f=m$  \cite[Example $6.3$]{KarenLocal}. 
Then, $m\subset J_{F}(f)$ by Proposition \ref{DefFlag}.

We have that $R/fR$ is $F$-pure if and only if
$p\equiv1\mbox{ mod } 3.$ 
We have that $(m^{[p]}:f^{p-1})=m$ if $p\equiv1\mbox{ mod } 3, $ and 
$(m^{[p]}:f^{p-1})=R$ if $p\equiv1\mbox{ mod } 2.$
Hence,
\[
J_{F}(f)=\begin{cases}
R & p\equiv2\mbox{ mod } 3\\
m & p\equiv1\mbox{ mod } 3.
\end{cases}
\]
\end{Ex}

\begin{Ex}
Let $R=K[x_1,\ldots,x_n],$ where $K$ is an $F$--finite field of characteristic $p>0.$
Let $f=a_1x^{d_1}_1+\ldots+a_nx^{d_n}_n,$ be such that $a_\neq 0.$
We have that $R/fR$ has an isolated singularity at the maximal ideal $m=(x_1,\ldots,x_n).$

If $\frac{1}{d_1}+\ldots+\frac{1}{d_n}= 1$
and $(p-1)/d_1$ is an integer for every $i$,
then $R/fR$ is $F$-pure for $p\gg 0$  \cite[Theorem $3.1$]{Daniel}
and not $F$-regular  \cite[Theorem $3.1$]{Donna} because $f^{p-1}$ is congruent to  
$c (x_1 \cdots x_n)^{p^e-1}$ module $m^{[p^e]}$ for a nonzero element  $c\in K$. Hence, $J_F(f)=R$ for $p\gg 0$ by Proposition \ref{FpureIsolated}.
\end{Ex}

\begin{Rem}
Let $R=K[x_1,\ldots, x_n]$ be a polynomial ring and $f\in R$ be such that $R/fR$ is reduced.
We can obtain $J_F(f)$ from $\tau(R/fR)$ by Proposition \ref{DefFlag}. In the case where $n>3,$
$f=x^d_1+\ldots x^d_n$ and $d$ is not divided by the characteristic of $K,$ there is an algorithm to compute the test ideal of 
$R/fR$ \cite{diagonal}. Therefore, there is an algorithm to compute $J_F(f).$
\end{Rem}
\begin{Ex}
Let $R=K[x_1,\ldots,x_n],$ where $K$ is a field of characteristic $p>0.$
Let $f=x^{d}_1+\ldots +x^{d}_n.$ 
This examples are based in computations done by McDermott \cite[Example $11,$ $12$ and $13$]{diagonal}.

If $p=2,n=5$ and $d=5,$ 
$$
\tau_f=(x^2_ix_j)_{1\leq i,j\leq 5}.
$$
Then, $(x^2_1,x^2_2,x^2_3,x^2_4,x^2_5,x_1x_2x_3x_4x_5)R=(\tau_f^{[2]}:f)$ and $R=(\tau_f^{[4]}:f^3)$. Hence, $J_F(f)=R.$

If $p=3,n=4$ and $d=7,$ 
$$\tau_f=(x^2_i x^2_j)_{1\leq i,j\leq 4}.$$
Then $R=(\tau_f^{[3]}:f^2)$ and  $J_F(f)=R.$

If $p=7,n=5$ and $d=4,$ 
$$\tau_f=(x_1,\ldots,x_5)R.$$
Then $R=(\tau_f^{[7]}:f^6)$ and  $J_F(f)=R.$

\end{Ex}

\section{Further consequences and relations}\label{SecRelations}
Using ideas and techniques developed in the previous sections, we obtain relations among test ideals, generalized test ideals and generalized Lyubeznik numbers for hypersurfaces.

\begin{Notation}\label{Notation5}
Throughout this section $R$ denotes a ring 
essentially of finite type over an $F$--finite
local ring.
Let $f\in R$ denote an element such that $R/fR$ is reduced.
Let $\pi:R\to R/fR$ be the quotient morphism. $\tau_f$ denotes $\pi^{-1}(\tau(R/fR)),$ the pullback of the test ideal of $R/fR.$
\end{Notation}
\begin{Rem}\label{CompareTestIdeals}
We have that $\tau(f^{1-\epsilon})\frac{1}{f}$ is the minimal root for $R_f$ as an $F$-module.
Then $$\tau(f^{1-\epsilon})/fR\FDer{f^{q-1}} \tau(f^{1-\epsilon})^{[q]}/f^qR$$ generates $R_f/R$
(see Proposition \ref{TestIdealMinimalRoot} or  \cite[Proposition $3.5$]{GammaSheaves}.
We have that $\tau(R/fR)=\tau_f/f$ is the minimal root of $\Min_F(f)$  \cite[Proposition $4.5$ and Theorem $4.6$]{Manuel}.
Therefore, $\tau_f\frac{1}{f}\subset \tau(f^{1-\epsilon})\frac{1}{f}\subset R_1/R.$ Then,
$$\tau_f\subset \tau(f^{1-\epsilon}).$$ This containment can also be obtained using $F$-adjunction \cite{Takagi,Karl}.

Suppose that $R$ is a finitely generated polynomial ring over an $F$-finite field and $f$ is a polynomial.
Since the Jacobian ideal of $f$ is contained in $\tau_f$ when $R/fR$ is reduced \cite{HHCharZero}, the previous containment of test ideals
gives implies $J(f)\subset\tau(f^{1-\epsilon}).$
The containment of the Jacobian ideal in the test ideal $\tau(f^{1-\epsilon})$ was proved using different techniques \cite{KLZ} and 
still holds when $R/fR$ is not reduced. 
\end{Rem}
We assume that $(R,m,K)$ is a local ring and $F$-finite. Vassilev proved that if $I\subset R$ 
is an ideal and $\tau_I$ is the pullback of the test ideal of 
$R/I,$ then $(I^{[p]}:I)\subset (\tau_I^{[p]}:\tau_I)$ \cite{TestIdeals}. In particular, if $R/I$ is $F$-pure, then $R/\tau_I$ is also $F$-pure.
Using this criteria, she showed that if $R/I$ is $F$-pure, there is a sequence of ideals 
\begin{equation}\label{EqV}
 I=\tau_0\subset \tau_1\subset\ldots
\end{equation}
such that $\tau_{i+1}$ is the pull back of the test ideal of $R/\tau_i.$
The sequence in \ref{EqV} stabilizes when $\tau_i=R,$ because if $\tau_i\neq 0,$ $ R/\tau_i,$ then is an $F$-finite reduced local ring 
and its test ideal is not zero \cite[Theorem 6.1]{HoHu2}.
We  apply this ideas to bound the length of $M_1/N_1=R_f/R$ as an $F$-module.
\begin{Lemma}\label{TestFMod}
Suppose that $(R,m,K)$ is a local ring and $F$-finite and $R/fR$ is $F$-pure.
Let $$0\subset f=\tau_0\subset \tau_1\subset\ldots\subset \tau_\ell=R$$ be the sequence of pullback of test ideals for $fR$ defined  in \ref{EqV}. If $\tau_{i+1}\neq \tau_i,$
then $N_{\tau_{i+1}}\neq N_{\tau_i}.$
\end{Lemma}
\begin{proof}
We have that $f^{p-1}\in (f^p:f)\subset (\tau^{[p]}_j:\tau_j)$ for every $j.$
We will proceed by contradiction. Suppose that 
$\tau_{i+1}\neq \tau_i,$
and $N_{\tau_{i+1}}= N_{\tau_i}.$
Then there exists $e$ such that $f^{p^e-1}\tau_{i+1}\subset \tau_i^{[p^e]}$ by Lemma \ref{GeneratesSame}.
Since $R/\tau_j$ is $F$-pure for every $j,$ we have that both $\tau_i$ and $\tau_{i+1}$ are radical ideals.
Therefore, we can choose a minimal prime $Q$ of $\tau_i$ such that $(\tau_{i+1})_Q=R_Q.$
Hence $f^{p^e-1}(\tau_{i+1})_Q=f^{p^e-1}R_Q\subset (\tau_i)_Q^{[p^e]}\subset (QR_Q)^{[p^e]},$ and we obtain a contradiction because $R_Q/fR_Q$ is $F$-pure
\end{proof}
\begin{Prop}\label{FlagV}
Suppose that $(R,m,K)$ is local and $R/fR$ is $F$-pure. 
If 
$$0\subset fR=\tau_0\subset \tau_1\subset\ldots\subset \tau_\ell =R$$
is the sequence of ideals for $fR$ defined in \ref{EqV}, then
$$\ell \leq \Length_{F-\hbox{{\tiny mod}}} R_f/R$$
\end{Prop}
\begin{proof}
We have that every pullback of test ideal $\tau_i$ defines a different $F$-module by Lemma \ref{TestFMod}, which proves the proposition. 
\end{proof}

The Lyubeznik numbers are invariants of a local ring of equal characteristic   \cite[Theorem/Definition $4.1$]{LyuDmodules}. 
The first author and Witt defined a generalization these invariants using the theory of $D$-modules and local cohomology \cite[Definition $4.3$]{NuWi}.
We only recall the definition of hypersurfaces; 
however,  the generalized Lyubeznik numbers are defined for any local ring of equal characteristic and  a sequence of integers and ideals.

Suppose that $R=K[[x_1,\ldots,x_n]]$
Let $d=\dim(R/fR)=n-1$ and $L$ be a coefficient field of $R/fR.$
We define the generalized Lyubeznik $\lambda^d_0(R/fR;L)$ by $\Length_{D(R,L)} R_f/R.$
This number is well defined and depends only on $R/fR$ and $L.$
\begin{Cor}\label{BoundLyuNum}
Suppose that $R=K[[x_1,\ldots,x_n]],$ where $K$ is an $F$-finite field, and that $R/fR$ is $F$-pure. 
Let $L$ be a coefficient field of $R/fR.$
If 
$$0\subset fR=\tau_0\subset \tau_1\subset\ldots\subset \tau_\ell=R$$
is the sequence of ideals in \ref{FlagV},
then
$$
\ell\leq\lambda^{\dim(R/fR)}_0(R/fR;L).
$$
\end{Cor}
\begin{proof}
We have that
$$
\ell \leq \Length_{F-\hbox{{\tiny mod}}} R_f/R \leq \Length_{D(R,\ZZ)} R_f/R 
 \leq \Length_{D(R,L)} R_f/R = \lambda^{\dim(R/fR)}_0(R/fR;L).
$$
\end{proof}
\begin{Rem}
If $R/fR$ is $F$-pure and $\lambda^{\dim(R/fR)}_0(R/fR;L)=1,$ then $R/fR$ is $F$-regular \cite{Manuel,NuWi}.
In addition, if $R/fR$ is $F$-pure and $K$ perfect, then
$\lambda^{\dim(R/fR)}_0(R/fR;L)=1$ if and only if $R/fR$ is $F$-regular \cite{Manuel,NuWi}.
Then, Corollary \ref{BoundLyuNum} is telling us that for $F$-pure hypersurfaces  $\lambda^{\dim(R/fR)}_0(R/fR;L)$
measures how far is $R/fR$ from being $F$-regular. 
\end{Rem}

\begin{Prop}\label{UpperLengthFMod}
Suppose that $(R,m,K)$ is local and $F$-finite, and that $R/fR$ is reduced and $\Length_R \tau(f^{1-\epsilon})/\tau_f$ is finite. 
Let $f=f_1\ldots f_\ell$ be a factorization of $f$ into irreducible elements.
Then,
$$
\Length_{F-\hbox{{\tiny mod}}} R_f/R
\leq\Length_R(\tau(f^{1-\epsilon})/\tau_f)+\ell
$$
\end{Prop}
\begin{proof}
We have that $\min_F(f)=\min_F(f)\oplus\ldots\oplus \min_F(f_\ell).$
In addition, the length of $\min_F(f)$ as a $F$-module is $\ell.$
We recall that $\tau(f^{1-\epsilon})\frac{1}{f}$ is a root for $R_f/R$ by Proposition \ref{MinimalRoot}
and $\tau_f\frac{1}{f}$ is the minimal root for $\Min_F(f).$ Therefore, for every $F$-submodule  $N\subset R_f/R$ that contains 
$\Min_F(f)$, $N\cap \tau(f^{1-\epsilon})\frac{1}{f}$
is a root for $N$ and it must contain $\tau_f\frac{1}{f}$.
Then for every strict ascending chain of $F$-modules $\Min_F(f)=N_0\subset N_1\subset \ldots\subset N_t=R_f/R,$
$t\leq\Length_R\tau(f^{1-\epsilon})/\tau_f.$ Hence
$$
 \Length_{F-\hbox{{\tiny mod}}} R_f/R
\leq\Length_R(\tau(f^{1-\epsilon})/\tau_f)+\ell.
$$
\end{proof}
\begin{proof}[Proof of Theorem \ref{MainBounds}]
This is consequence of Propositions \ref{FlagV} and \ref{UpperLengthFMod}
and Corollary \ref{BoundLyuNum}.
\end{proof}
 
\section*{Acknowledgments}
We thank Josep \`Alvarez Montaner, Ang\'elica Benito,
Daniel J. Hern\'andez, Mel Hochster, Mircea Musta\c{t}\u{a}, Emily E. Witt 
and Wenliang Zhang for insightful mathematical conversations related to this work.
We thank Manuel Blickle, Karl schwede and Kevin Tucker for poitnting out missed references.
We also thank Alexis Cook for her careful reading of this manuscript.
The first author thanks the National Council of Science and Technology of Mexico (CONACyT) for its support through grant $210916.$ 

\bibliographystyle{alpha}
\bibliography{References}

\vspace{.3cm}
\small{
{\sc Department of Mathematics, University of Michigan, Ann Arbor, MI $48109$-$1043,$ USA.}

{\it Email address:}  \texttt{luisnub@umich.edu}

\vspace{.3cm}

{\sc Department of Mathematics, University of Michigan, Ann Arbor, MI $48109$-$1043,$ USA.}

{\it Email address:}  \texttt{juanfp@umich.edu}

}

\end{document}